\documentclass[11pt,a4paper]{article}

\usepackage{amsmath}
\usepackage{amssymb}
\usepackage{amsthm}
\usepackage{authblk}
\usepackage{cite}
\usepackage{geometry}
\usepackage[english]{babel}
\usepackage[T1]{fontenc}
\usepackage{graphicx}
\usepackage{mathtools}
\usepackage{subfig}
\usepackage[utf8]{inputenc}
\usepackage{pgf}
\usepackage{url}

\newtheorem{theo}{Theorem}
\newtheorem{lemm}{Lemma}
\newtheorem{prop}{Proposition}
\theoremstyle{definition}
\newtheorem{mydef}{Definition} 
\theoremstyle{definition}

\numberwithin{equation}{section}

\title{Adaptive energy preserving methods for\\ partial differential equations}
\author{Sølve Eidnes \and Brynjulf Owren \and Torbjørn Ringholm}
\author{Torbjørn Ringholm, Brynjulf Owren, Sølve Eidnes}
\begin{document}
\maketitle

\begin{abstract}
\noindent A framework for constructing integral preserving numerical schemes for time-dependent partial differential equations on non-uniform grids is presented. The approach can be used with both finite difference and partition of unity methods, thereby including finite element methods. The schemes are then extended to accommodate $r$-, $h$- and $p$-adaptivity. To illustrate the ideas, the method is applied to the Korteweg--de Vries equation and the sine-Gordon equation. Results from numerical experiments are presented.
\end{abstract}

\section{Introduction}

Courant, Friedrichs and Lewy introduced difference schemes with conservation properties in \cite{courant28udp}, where a discrete conservation law for a finite difference approximation of the wave equation was derived. Their methods are often called energy methods \cite{furihata01fds} or energy-conserving methods \cite{li95fdc}, although the conserved quantity is often not energy in the physical sense. The primary motivation for developing conservative methods was originally to devise a norm that could guarantee global stability. This was still an objective, in addition to proving existence and uniqueness of solutions, when the energy methods garnered newfound interest in the 1950s and 1960s, resulting in new developments such as generalizations of the methods and more difference schemes, summarized by Richtmyer and Morton in \cite{richtmyer67dmf}.
In the 1970s, the motivation behind studying schemes that preserve invariant quantities changed, as the focus shifted to the conservation property itself. Li and Vu-Quoc presented in \cite{li95fdc} a historical survey of conservative methods developed up to the early 1990s. They state that this line of work is motivated by the fact that in some situations, the success of a numerical solution will depend on its ability to preserve one or more of the invariant properties of the original differential equation. In addition, as noted in \cite{christiansen11tis,hairer06gni}, there is the general idea that transferring more of the properties of the original continuous dynamical system over to a discrete dynamical system may lead to a more accurate numerical approximation of the solution, especially over long time intervals.
In recent years, there has been a greater interest in developing systematic techniques applicable to larger classes of differential equations. Hairer, Lubich and Wanner give in \cite{hairer06gni} a presentation of geometric integrators for differential equations, i.e. methods for solving ordinary differential equations (ODEs) that preserve a geometric structure of the system. Examples of such geometric structures are symplectic structures, symmetries, reversing symmetries, isospectrality, Lie group structure, orthonormality, first integrals, and other  invariants, such as volume and invariant measure.

In this paper we will be concerned with the preservation of first integrals of PDEs. From the ODE literature we find that the most general methods for preserving first integrals are tailored schemes, in the sense that the vector field of the ODE does not by itself provide sufficient information, so the schemes make explicit use of the first integral. An obvious approach in this respect is projection, where the solution is first advanced using any consistent numerical scheme and then this approximation is projected  onto the appropriate level set of the invariant.
In the same class of tailored methods one also has the discrete gradient methods, 
usually attributed to  Gonzalez  \cite{gonzalez96tia}.  For the subclass of canonical Hamiltonian systems, the energy can be preserved by means of a general purpose method called the averaged vector field method, see e.g. \cite{quispel08anc}.

The notion of discrete gradient methods for ordinary differential equations has a counterpart for partial differential equations called the discrete variational derivative method. Such schemes have been developed since the late 1990s in a number of articles by Japanese researchers such as Furihata, Matsuo, Sugihara, and Yaguchi. A relatively recent account of this work can be found in the monograph \cite{furihata11dvd}.
More recently, the development of integral preserving schemes for PDEs has been systematised and eased, in particular by using the aforementioned tools from ordinary differential equations, see for instance \cite{celledoni12per,dahlby11agf}.
Most of the schemes one finds in the literature are based on a finite difference approach, and usually on fixed, uniform grids.
There are however some exceptions.
Yaguchi, Matsuo and Sugihara presented in \cite{yaguchi10aeo, yaguchi12tdv} two different discrete variational derivative methods on fixed, non-uniform grids, specifically defined for certain classes of PDEs. Non-uniform grids are of particular importance for multidimensional problems, since the use of uniform grids will greatly restrict the types of domains possible to discretize. Another important consequence of being able to use non-uniform grids is that it allows for the use of time-adaptive spatial meshes for solving partial differential equations. Adaptive energy preserving schemes for the Korteweg--de Vries and Cahn--Hilliard equations have been developed recently \cite{MM_Miyatake_Matsuo} by Miyatake and Matsuo. The main objective of this paper is to propose a general framework for numerical methods for PDEs that combine mesh adaptivity with first integral conservation.

Several forms of adaptive methods exist; they can roughly be categorized as $r$-, $h$- and $p$-adaptive. When applying $r$-adaptivity, one keeps the number of degrees of freedom constant while modifying the mesh at each time step to e.g. cluster in problematic areas such as boundary layers or to follow wave fronts. When applying the Finite Difference Method (FDM) or the Finite Element Method (FEM), moving mesh methods may be used for $r$-adaptivity, some examples of which may be found in \cite{MM_Huang_Russell,MM_Zegeling,MM_Baines}. When using Partition of Unity Methods (PUM) (and in particular when using FEM), $h$- and $p$-adaptivity relate to adjusting the number of elements and the basis functions used on the elements, respectively.   For PUM methods there exist strategies for $h$- and $p$-adaptivity based both on a priori and a posteriori error analysis \cite{HPFEM_Babuska_Guo}. Common to all of these strategies is that, based on estimated function values in preceding time steps, one can suggest improved discretization parameters for the next time step. In the FDM approach, these discretization parameters consist of the mesh points $\mathbf{x}$, while in the PUM approach the parameters encompass information about both the mesh and the basis functions. We will, in general, denote a collection of discretization parameters by $\mathbf{p}$, and assume that the discretization parameters are changed separately from the degrees of freedom $\mathbf{u}$ of the problem when using adaptive methods. That is, starting with an initial set of discretization parameters $\mathbf{p}^0$ and initial values $\mathbf{u}^0$, one would first decide upon $\mathbf{p}^1$ before calculating $\mathbf{u}^1$, then finding $\mathbf{p}^2$, then $\mathbf{u}^2$, etc., in a decoupled fashion.

A first integral of a PDE is a functional $\mathcal{I}$ on an infinite-dimensional function space, whereas the numerical methods considered here will reduce the problem to a finite-dimensional setting. Therefore, we cannot preserve the exact value of the first integral; instead, we will preserve a consistent approximation to the first integral, $\mathcal{I}_\mathbf{p}(\mathbf{u})$. The approximation will be dependent on the discretization parameters $\mathbf{p}$ and, since adaptivity alters the discretization parameters, we will therefore aim to preserve the value of the approximated first integral across all discretization parameters, i.e. we will require that $\mathcal{I}_{\mathbf{p}^{n+1}}(\mathbf{u}^{n+1})$ = $\mathcal{I}_{\mathbf{p}^{n}}(\mathbf{u}^{n})$. Here, and in the following, superscripts denote time steps unless otherwise specified.  

In this article, we will present a method for developing adaptive numerical schemes that conserve an approximated first integral. In Section 2, the PDE problem is stated, and two classes of first integral preserving methods using arbitrary yet constant discretization parameters are presented; one using an FDM approach and the other a PUM approach for spatial discretization. A connection to previously existing methods is then established. In Section 3, we present a way of adding adaptivity to the methods from Section 2 and the modifications needed to retain the first integral preservation property, before showing that certain projection methods form a subclass of the methods thus obtained. Section 4 contains examples of the application of the methods to two PDEs and numerical results pertaining to the quality of the numerical solutions as compared to a standard implicit method.

\section{Spatial discretization with fixed mesh}

\subsection{Problem statement}
Consider a partial differential equation
\begin{align}
u_t = f(\mathbf{x},u^J), \qquad \mathbf{x} \in \Omega \subseteq \mathbb{R}^d,\quad u \in \mathcal{B} \subseteq L^2, \label{eq:pde_pure}
\end{align}
where $u^J$ denotes $u$ itself and its partial derivatives of any order with respect to the spatial variables $x_1,....,x_d$. We shall not specify the space $\mathcal{B}$ further, but assume that it is sufficiently regular to allow all operations used in the following. For ease of reading, all $t$-dependence will be suppressed in the notation wherever it is irrelevant. Also, from here on, square brackets are used to denote dependence on a function and its partial derivatives of any order with respect to the independent variables $t$ and $ x_1,...,x_d$. We recall the definition of the \textit{variational derivative} of a functional $H[u]$ as the function $\frac{\delta H}{\delta u}[u]$ satisfying
\begin{align}
\left\langle \dfrac{\delta H}{\delta u}[u], v \right\rangle_{L^2} = \dfrac{\mathrm{d}}{\mathrm{d}\epsilon} \bigg |_{\epsilon = 0} H[u+\epsilon v] \quad \forall v \in \mathcal{B},\label{eq:Idef}
\end{align}
and define a \textit{first integral} of (\ref{eq:pde_pure}) to be a functional $\mathcal{I}[u]$ satisfying
\begin{align*}
\left\langle \dfrac{\delta \mathcal{I}}{\delta u}[u], f(\mathbf{x},u^J) \right\rangle_{L^2}  = 0, \quad \forall u \in \mathcal{B}.
\end{align*}
We may observe that $\mathcal{I}[u]$ is preserved over time, since this implies
\begin{align*}
\dfrac{\mathrm{d}\mathcal{I}}{\mathrm{d}t} = \left\langle \dfrac{\delta \mathcal{I}}{\delta u}[u], \dfrac{\partial u}{\partial t}\right\rangle_{L^2} = 0.
\end{align*}
Furthermore, we may observe that if there exists some operator $S(\mathbf{x},u^J)$, skew-symmetric with respect to the $L^2$ inner product, such that 
\begin{align*}
f(\mathbf{x},u^J) = S(\mathbf{x},u^J)\dfrac{\delta \mathcal{I}}{\delta u}[u],
\end{align*}
then $\mathcal{I}[u]$ is a first integral of (\ref{eq:pde_pure}), and we can state (\ref{eq:pde_pure}) in the form
\begin{align}
u_t = S(\mathbf{x}, u^J)\dfrac{\delta \mathcal{I}}{\delta u}[u]. \label{eq:pde_var}
\end{align}
This can be considered as the PDE analogue of an ODE with a first integral, in which case we have a system
\begin{align}
\dfrac{\mathrm{d}\mathbf{u}}{\mathrm{d}t} = S(\mathbf{u})\nabla_\mathbf{u} I(\mathbf{u}), \label{eq:ODE_form}
\end{align}
where $S(\mathbf{u})$ is a skew-symmetric matrix \cite{mclachlan99giu}. The gradient is defined as usual, but for clarity in later use we have added a subscript to specify that it is a vector of partial derivatives with respect to the coordinates of $\mathbf{u}$. Note that Hamiltonian equations are contained of this class of ODEs. For such differential equations, there exist numerical methods preserving the first integral $I(\mathbf{u})$, for instance the discrete gradient methods, which are of the form
\begin{align*}
\dfrac{\mathbf{u}^{n+1} - \mathbf{u}^n}{\Delta t} = \bar{S}(\mathbf{u}^n,\mathbf{u}^{n+1}) \overline{\nabla}I(\mathbf{u}^n,\mathbf{u}^{n+1}),
\end{align*}
where $\bar{S}(\mathbf{u}^{n},\mathbf{u}^{n+1})$ is a consistent skew-symmetric time-discrete approximation to $S(\mathbf{u})$ and $\overline{\nabla}I(\mathbf{v},\mathbf{u})$ is a discrete gradient of $I(\mathbf{u})$, i.e. a function satisfying
\begin{align}
(\overline{\nabla}I(\mathbf{v},\mathbf{u}))^T(\mathbf{u} - \mathbf{v}) &= I(\mathbf{u}) - I(\mathbf{v}),\label{eq:DG1}\\
\overline{\nabla}I(\mathbf{u},\mathbf{u}) &= \nabla_\mathbf{u} I(\mathbf{u}).\label{eq:DG2}
\end{align}
There are several possible choices of discrete gradients available, one of which is the Average Vector Field (AVF) discrete gradient \cite{celledoni12per}, given by
\begin{align*}
\overline{\nabla}I(\mathbf{v},\mathbf{u}) = \int \limits_0^1 \nabla_\mathbf{u} I(\xi \mathbf{u} + (1 - \xi)\mathbf{v}) \mathrm{d}\xi,
\end{align*}
which will be used for numerical experiments in the final chapter. Our approach to solving (\ref{eq:pde_pure}) on non-uniform grids is based upon considering the PDE in the form (\ref{eq:pde_var}), reducing it to a system of ODEs of the form (\ref{eq:ODE_form}) and applying a discrete gradient method. This is done by finding a discrete approximation $\mathcal{I}_{\mathbf{p}}$ to $\mathcal{I}$ and using this to obtain a discretization in the spatial variables, which is achieved through either a finite difference approach or a variational approach.

\subsection{Finite difference method}
\label{sec:FDM}
In the finite difference approach, we restrict ourselves to obtaining approximate values of $u$ at the grid points $\mathbf{x}_0, ... , \mathbf{x}_M$, which can be interpreted as quadrature points with some associated nonzero quadrature weights $\kappa_0, ..., \kappa_M$. The grid points constitute the discretization parameters $\mathbf{p}$. We can then approximate the $L^2$ inner product by quadrature to arrive at a weighted inner product:
\begin{align*}
\left\langle u,v \right\rangle_{L^2} = \int \limits_{\Omega}u(\mathbf{x})v(\mathbf{x}) \mathrm{d}x \simeq \sum\limits_{i=0}^{M}\kappa_i u(\mathbf{x}_i)v(\mathbf{x}_i) = \mathbf{u}^T D(\kappa) \mathbf{v} = \left\langle \mathbf{u},\mathbf{v} \right\rangle_{\kappa},
\end{align*}
where $D(\kappa) = \mathrm{diag}(\kappa_0,...,\kappa_M)$. Assume that there exists a consistent approximation $\mathcal{I}_{\mathbf{p}}(\mathbf{u})$ to the functional $\mathcal{I}[u]$, dependent on the values of $u$ at the points $\mathbf{x}_i$. Then, we can characterize the discretized variational derivative by asserting that 
\begin{align*}
\left\langle \dfrac{\delta \mathcal{I}_{\mathbf{p}}}{\delta \mathbf{u}}(\mathbf{u}), \mathbf{v} \right\rangle_{\kappa} = \dfrac{\mathrm{d}}{\mathrm{d}\epsilon} \bigg |_{\epsilon = 0} \mathcal{I}_{\mathbf{p}}(\mathbf{u}+\epsilon \mathbf{v}) \quad \forall \mathbf{v} \in \mathbb{R}^{M+1} ,
\end{align*}
meaning
\begin{align*}
\left(\dfrac{\delta \mathcal{I}_{\mathbf{p}}}{\delta \mathbf{u}}(\mathbf{u})\right)^T D(\kappa) \mathbf{v} = (\nabla_\mathbf{u} \mathcal{I}_{\mathbf{p}} (\mathbf{u}))^T \mathbf{v} \quad \forall \mathbf{v} \in \mathbb{R}^{M+1} ,
\end{align*}
from which we conclude that
\begin{align}
\dfrac{\delta \mathcal{I}_{\mathbf{p}}}{\delta \mathbf{u}}(\mathbf{u}) = D(\kappa)^{-1} \nabla_\mathbf{u} \mathcal{I}_{\mathbf{p}} (\mathbf{u}).\label{eq:vdg}
\end{align}
Using this as a discretization of $\frac{\delta \mathcal{I}}{\delta u}[u]$ and approximating $S(\mathbf{x}, u^J)$ by a matrix $S_{d}(\mathbf{u})$, skew-symmetric with respect to $\langle \cdot,\cdot \rangle_\kappa$, we obtain a discretization of (\ref{eq:pde_var}) as:
\begin{align}
\dfrac{\mathrm{d}\mathbf{u}}{\mathrm{d}t} = S_{\mathbf{p}}(\mathbf{u}) \nabla_\mathbf{u} \mathcal{I}_{\mathbf{p}}(\mathbf{u}),
\label{eq:ODEs}
\end{align}
where $S_{\mathbf{p}}(\mathbf{u}) = S_{d}(\mathbf{u})D(\kappa)^{-1} $. This system of ODEs is of the form (\ref{eq:ODE_form}), since 
\begin{align*}
S_{\mathbf{p}}(\mathbf{u})^T &= (S_{d}(\mathbf{u})D(\kappa)^{-1})^T\\
 &= D(\kappa)^{-1} S_{d}(\mathbf{u})^T D(\kappa) D(\kappa)^{-1}\\
& = - D(\kappa)^{-1} D(\kappa) S_{d}(\mathbf{u}) D(\kappa)^{-1}\\
 & = -S_{d}(\mathbf{u}) D(\kappa)^{-1} \\
 &= -S_{\mathbf{p}}(\mathbf{u}).
\end{align*}
This allows us to apply first integral preserving methods for systems of ODEs to solve the spatially discretized system. For example, we may consider using a discrete gradient $\overline{\nabla}\mathcal{I}_{\mathbf{p}}$, and a skew-symmetric, time-discrete approximation $S_{\mathbf{p}}(\mathbf{u}^{n},\mathbf{u}^{n+1})$ to $ S_{\mathbf{p}}(\mathbf{u})$, where $\mathbf{u}^n = \mathbf{u}(t_n)$, $t_n = n \Delta t$. Then, the following scheme will preserve the approximated first integral $\mathcal{I}_{\mathbf{p}}$ in the sense that $\mathcal{I}_{\mathbf{p}}(\mathbf{u}^{n+1}) = \mathcal{I}_{\mathbf{p}}(\mathbf{u}^n)$:
\begin{align}
\dfrac{\mathbf{u}^{n+1} - \mathbf{u}^n}{\Delta t} = S_{\mathbf{p}}(\mathbf{u}^{n},\mathbf{u}^{n+1}) \overline{\nabla}\mathcal{I}_{\mathbf{p}}(\mathbf{u}^{n},\mathbf{u}^{n+1}). \label{eq:DG}
\end{align}

\subsection{Partition of unity method}
\label{sec:PUM}
One may also approach the problem of spatially discretizing the PDE through the use of variational methods such as the Partition of Unity Method (PUM) \cite{PUM_Melenk_CMAM}, which generalizes the Finite Element Method (FEM). Here, the variational structure of the functional derivative can be utilized in a natural way, such that one avoids having to approximate $S(\mathbf{x}, u^J)$. We begin by stating a weak form of (\ref{eq:pde_var}). Then, the problem consists of finding $u \in \mathcal{B}$ such that
\begin{align}
\left\langle u_t,v \right\rangle_{L^2} = \left\langle S(\mathbf{x}, u^J)\dfrac{\delta \mathcal{I}}{\delta u}[u],v \right\rangle_{L^2} = - \left\langle \dfrac{\delta \mathcal{I}}{\delta u}[u],S(\mathbf{x}, u^J) v \right\rangle_{L^2} \quad \forall v \in \mathcal{B}.
\label{eq:weakprob}
\end{align}
Employing a Galerkin formulation, we restrict the search to a finite dimensional subspace $\mathcal{B}^h = \mathrm{span}\{ \varphi_0, ... \varphi_M \} \subseteq \mathcal{B}$, and approximate $u$ by the function
\begin{align*}
u^h(x,t) = \sum_{i=0}^M u_i(t) \varphi_i(x).
\end{align*}
We denote by $\mathbf{p}$ the collection of discretization parameters defining $\mathcal{B}^h$; this includes information about mesh points, element types and shapes of basis functions. Furthermore, we define the canonical mapping $\Phi_{\mathbf{p}}:  \mathbb{R}^{M+1} \rightarrow \mathcal{B}^h $ given by 
\begin{align}
\Phi_{\mathbf{p}}(\mathbf{u}) = \sum_{i=0}^M u_i \varphi_i, \label{eq:phi}
\end{align}
and the discrete first integral $\mathcal{I}_{\mathbf{p}}$ by
\begin{align*}
\mathcal{I}_{\mathbf{p}}(\mathbf{u}) = \mathcal{I}(\Phi_{\mathbf{p}}(\mathbf{u})). 
\end{align*}
The following lemma will prove useful later in the construction of the method:
\begin{lemm} \label{lem:PUM}
For any $u^h, v \in \mathcal{B}^h$,
\begin{align*}
\dfrac{\mathrm{d}}{\mathrm{d}\epsilon}\bigg |_{\epsilon = 0}\mathcal{I}(u^h + \epsilon v)  = (\nabla_\mathbf{u} \mathcal{I}_{\mathbf{p}}(\mathbf{u}))^T \mathbf{v}.
\end{align*}
\end{lemm}
\begin{proof}
\begin{align*}
\dfrac{\mathrm{d}}{\mathrm{d}\epsilon}\bigg |_{\epsilon = 0}\mathcal{I}(u^h + \epsilon v) &= \dfrac{\mathrm{d}}{\mathrm{d}\epsilon}\bigg |_{\epsilon = 0}\mathcal{I}(\Phi_{\mathbf{p}}(\mathbf{u} + \epsilon \mathbf{v})) \\
&= \left\langle \dfrac{\delta \mathcal{I}}{\delta u}[\Phi_{\mathbf{p}}(\mathbf{u} + \epsilon \mathbf{v})] , \dfrac{\mathrm{d}}{\mathrm{d}\epsilon} \Phi_{\mathbf{p}}(\mathbf{u} + \epsilon \mathbf{v}) \right\rangle_{L^2}\bigg |_{\epsilon = 0} \\
&=  \left\langle \dfrac{\delta \mathcal{I}}{\delta u}[\Phi_{\mathbf{p}}(\mathbf{u} + \epsilon \mathbf{v})] , (\nabla_\mathbf{u} \Phi_{\mathbf{p}}(\mathbf{u} + \epsilon \mathbf{v}))^T\mathbf{v} \right\rangle_{L^2}\bigg |_{\epsilon = 0}\\
&=  \left\langle \dfrac{\delta \mathcal{I}}{\delta u}[\Phi_{\mathbf{p}}(\mathbf{u})] , (\nabla_\mathbf{u} \Phi_{\mathbf{p}}(\mathbf{u} ))^T\mathbf{v} \right\rangle_{L^2}\\
&= \sum_{i=0}^M v_i \left\langle \dfrac{\delta \mathcal{I}}{\delta u}[\Phi_{\mathbf{p}}(\mathbf{u})] , \frac{\partial}{\partial u_i}\Phi_{\mathbf{p}}(\mathbf{u} )\right\rangle_{L^2}\\
&= \sum_{i=0}^M v_i \frac{\partial}{\partial u_i}  \mathcal{I}[\Phi_{\mathbf{p}}(\mathbf{u})]
= \sum_{i=0}^M v_i \frac{\partial}{\partial u_i}  \mathcal{I}_{\mathbf{p}}(\mathbf{u}) = (\nabla_\mathbf{u} \mathcal{I}_{\mathbf{p}}(\mathbf{u}))^T \mathbf{v}.
\end{align*}
\end{proof}
\noindent We observe that for $u,v \in \mathcal{B}^h$, the $L^2$ inner product has a discrete counterpart:
\begin{align*}
\left\langle u,v \right\rangle_{L^2} = \sum_{i=0}^M  \sum_{j=0}^M u_i v_j \left\langle \varphi_i,\varphi_j \right\rangle_{L^2} = \mathbf{u}^T A \mathbf{v} = \left\langle \mathbf{u},\mathbf{v}  \right\rangle_{A} 
\end{align*}
with the symmetric positive definite matrix $A$ given by $A_{ij} = \left\langle \varphi_i,\varphi_j \right\rangle_{L^2}$. Note also that equation (\ref{eq:weakprob}) is satisfied in $\mathcal{B}^h$ if it is satisfied for all basis functions $\varphi_j$. The Galerkin form of the problem therefore consists of finding $u_i(t)$ such that
\begin{align}
\sum_{i=0}^M \dfrac{\mathrm{d} u_i}{\mathrm{d} t} \left\langle \varphi_i,\varphi_j \right\rangle_{L^2} = - \left\langle \dfrac{\delta \mathcal{I}}{\delta u}[u^h],S(\mathbf{x}, u^{h,J}) \varphi_j \right\rangle_{L^2} \quad \forall j \in \{0,...,M\}.
\label{eq:galerkin}
\end{align}
This weak form is rather unwieldy and does not give rise to a system of the form (\ref{eq:ODE_form}), so in order to make further progress, we consider the projection of $\frac{\delta \mathcal{I}}{\delta u}[u^h]$ onto $\mathcal{B}^h$:
\begin{align*}
\dfrac{\delta \mathcal{I}}{\delta u}^h[u^h] =\sum_{i=0}^M w_i^h[u^h] \varphi_i(x) =\sum_{i=0}^M w_i(\mathbf{u}) \varphi_i(x),
\end{align*}
where $w_i(\mathbf{u}) = w_i^h[\Phi(\mathbf{u})] = w_i^h[u^h] $ are coefficients that will be characterized later. Replacing $\frac{\delta \mathcal{I}}{\delta u}[u^h]$ by its projection in (\ref{eq:galerkin}) gives the approximate weak form:
\begin{align*}
\sum_{i=0}^M\dfrac{\mathrm{d} u_i}{\mathrm{d} t}\left\langle \varphi_i,\varphi_j \right\rangle_{L^2} = - \sum_{i=0}^M w_i(\mathbf{u})\left\langle \varphi_i,S(\mathbf{x}, u^{h,J}) \varphi_j \right\rangle_{L^2} \quad \forall j \in \{0,...,M\}.
\end{align*}
Thus, we obtain a system of equations for the coefficients $u_i$:
\begin{align}
A \dfrac{\mathrm{d}\mathbf{u}}{\mathrm{d}t}= -B(\mathbf{u}) \mathbf{w}(\mathbf{u}),
\label{eq:FEM_ODE1}
\end{align}
with the skew-symmetric matrix $B(\mathbf{u})$ given by $B(\mathbf{u})_{ji} = \left\langle \varphi_i,S(\mathbf{x}, \Phi(\mathbf{u})^J) \varphi_j \right\rangle_{L^2}$. Furthermore, we may characterize the vector $\mathbf{w}(\mathbf{u})$ by the following argument:
\begin{align*}
\mathbf{w}(\mathbf{u})^TA\mathbf{v} = \left\langle \dfrac{\delta \mathcal{I}}{\delta u}^h[u^h] , v \right\rangle_{L^2} \!\!\!\!\!\! = \left\langle \dfrac{\delta \mathcal{I}}{\delta u} [u^h], v \right\rangle_{L^2} \!\!\!\!\!\! = \dfrac{\mathrm{d}}{\mathrm{d}\epsilon}\bigg |_{\epsilon = 0}\mathcal{I}(u^h + \epsilon v) = (\nabla_\mathbf{u} \mathcal{I}_{\mathbf{p}}(\mathbf{u}))^T \mathbf{v},
\end{align*}
where the last equality holds by Lemma \ref{lem:PUM}. This holds for all $\mathbf{v} \in \mathbb{R}^{M+1}$, and thus
\begin{align}
\mathbf{w}(\mathbf{u}) = A^{-1}\nabla_\mathbf{u} \mathcal{I}_{\mathbf{p}}(\mathbf{u}).
\label{eq:w_characterized}
\end{align}
Inserting (\ref{eq:w_characterized}) into (\ref{eq:FEM_ODE1}) and left-multiplying by $A^{-1}$, we are left with an ODE for the coefficients $u_i$:
\begin{align}
\dfrac{\mathrm{d}\mathbf{u}}{\mathrm{d}t} = S_{\mathbf{p}}(\mathbf{u}) \nabla_\mathbf{u} \mathcal{I}_{\mathbf{p}}(\mathbf{u}). \label{eq:PUM_ODE}
\end{align}
Here, $S_{\mathbf{p}}(\mathbf{u})  = -A^{-1} B(\mathbf{u})  A^{-1}$ is a skew-symmetric matrix, and the system is thereby of the form (\ref{eq:ODE_form}), meaning $\mathcal{I}_{\mathbf{p}}$ can be preserved numerically using e.g. discrete gradient methods as in equation (\ref{eq:DG}). 
\\
\\
\subsection{Discrete variational derivative methods}

Let us now define a general framework for the discrete variational derivative methods that encompass the methods presented by Furihata, Matsuo and coauthors in a number of publications including \cite{furihata01fds, furihata11dvd, yaguchi10aeo, yaguchi12tdv, furihata99fds}.
\begin{mydef}
Let $\mathcal{I}_{\mathbf{p}}$ be a consistent approximation to the functional $\mathcal{I}\left[u\right]$ discretized on $\mathbf{p}$ given by grid points $\mathbf{x}_i$ and quadrature weights $\kappa_i$, $i = 0,...,M$. Then $\frac{\delta \mathcal{I}_\mathbf{p}}{\delta(\mathbf{v},\mathbf{u})}(\mathbf{v},\mathbf{u})$ is a discrete variational derivative of $\mathcal{I}_\mathbf{p}(\mathbf{u})$ 
if it is a continuous function satisfying
\begin{eqnarray}
\left\langle \frac{\delta \mathcal{I}_\mathbf{p}}{\delta(\mathbf{v},\mathbf{u})}, \mathbf{u} - \mathbf{v} \right\rangle_\kappa &=& \mathcal{I}_\mathbf{p}(\mathbf{u}) - \mathcal{I}_\mathbf{p}(\mathbf{v}),
\label{dvdnug1} \\
\frac{\delta \mathcal{I}_\mathbf{p}}{\delta(\mathbf{u},\mathbf{u})} &=& \frac{\delta \mathcal{I}_\mathbf{p}}{\delta\mathbf{u}}\left(\mathbf{u}\right),
\label{dvdnug2}
\end{eqnarray}
and the discrete variational derivative methods for solving PDEs on the form ($\ref{eq:pde_var}$) are given by
\begin{align}
\dfrac{\mathbf{u}^{n+1} - \mathbf{u}^n}{\Delta t} = S_d(\mathbf{u}^{n},\mathbf{u}^{n+1}) \frac{\delta \mathcal{I}_\mathbf{p}}{\delta(\mathbf{u}^{n},\mathbf{u}^{n+1})}, \label{eq:DVDmet}
\end{align}
where $S_d(\mathbf{u}^{n},\mathbf{u}^{n+1})$ is a time-discrete approximation to $S_d(\mathbf{u})$, and itself skew-symmetric with respect to the inner product $\langle \cdot,\cdot \rangle_\kappa$.
\end{mydef}

\begin{prop}
A discrete gradient method (\ref{eq:DG}) applied to the system of ODEs (\ref{eq:ODEs}) or (\ref{eq:PUM_ODE}) is equivalent to a discrete variational derivative method as given by (\ref{eq:DVDmet}), with 
\begin{align*}
S_d(\mathbf{u}^{n},\mathbf{u}^{n+1}) = S_\mathbf{p}(\mathbf{u}^{n},\mathbf{u}^{n+1})D\left(\kappa\right),
\end{align*}
and the discrete variational derivative
\begin{align}
\frac{\delta \mathcal{I}_\mathbf{p}}{\delta (\mathbf{v},\mathbf{u})} = D(\kappa)^{-1}\overline{\nabla} \mathcal{I}_\mathbf{p}(\mathbf{v},\mathbf{u})
\label{dvddg}
\end{align}
satisfying (\ref{dvdnug1})-(\ref{dvdnug2}).
\end{prop}

\begin{proof}
Applying (\ref{eq:DG1}), we get that, for the discrete variational derivative defined by (\ref{dvddg}),
\begin{align*}
\left\langle \frac{\delta \mathcal{I}_\mathbf{p}}{\delta(\mathbf{v},\mathbf{u})}, \mathbf{u} - \mathbf{v} \right\rangle_\kappa &=  \left\langle D(\kappa)^{-1}\overline{\nabla} \mathcal{I}_\mathbf{p}(\mathbf{v},\mathbf{u}), \mathbf{u} - \mathbf{v} \right\rangle_\kappa \\
&= \left(D(\kappa)^{-1}\overline{\nabla} \mathcal{I}_\mathbf{p} \left(\mathbf{v},\mathbf{u}\right)\right)^{\text{T}} D(\kappa) \left(\mathbf{u} - \mathbf{v}\right)\\
&= \overline{\nabla} \mathcal{I}_\mathbf{p} \left(\mathbf{v},\mathbf{u}\right)^{\text{T}} (\mathbf{u} - \mathbf{v}) = \mathcal{I}_\mathbf{p}(\mathbf{u}) - \mathcal{I}_\mathbf{p}(\mathbf{v}),
\end{align*}
and hence (\ref{dvdnug1}) is satisfied. Furthermore, applying (\ref{eq:DG2}) and (\ref{eq:vdg}),
\begin{equation*}
\frac{\delta \mathcal{I}_\mathbf{p}}{\delta(\mathbf{u},\mathbf{u})} = D(\kappa)^{-1}\overline{\nabla} \mathcal{I}_\mathbf{p}(\mathbf{u},\mathbf{u}) = D(\kappa)^{-1}\nabla_\mathbf{u} \mathcal{I}_\mathbf{p}\left(\mathbf{u}\right) = \frac{\delta{\mathcal{I}_\mathbf{p}}}{\delta\mathbf{u}} \left(\mathbf{u}\right)
\end{equation*}
and (\ref{dvdnug2}) is also satisfied.
\end{proof}

Consequently, all discrete variational derivative methods as given by (\ref{eq:DVDmet}) can be expressed as discrete gradient methods on the system of ODEs (\ref{eq:ODEs}) or (\ref{eq:PUM_ODE}) obtained by discretizing (\ref{eq:pde_var}) in space, and vice versa.

\section{Adaptive discretization}

\subsection{Mapping solutions between parameter sets}
\label{sect:mapping}
Assuming that adaptive strategies are employed, one would obtain a new set of discretization parameters $\mathbf{p}$ at each time step. After such a $\mathbf{p}$ has been found, the solution using the previous parameters must be transferred to the new parameter set before advancing to the next time step. This transfer procedure can be done in either a preserving or a non-preserving manner. Let $\mathbf{p}^n$, $\mathbf{u}^n$, $\mathbf{p}^{n+1}$ and $\mathbf{u}^{n+1}$ denote the discretization parameters and the numerical values obtained at the current time step and next time step, respectively. Also, let $\hat{\mathbf{u}}$ denote the values of $\mathbf{u}^n$ transferred onto $\mathbf{p}^{n+1}$ by whatever means. We call the transfer operation preserving if $\mathcal{I}_{\mathbf{p}^{n+1}}(\hat{\mathbf{u}}) = \mathcal{I}_{\mathbf{p}^n}(\mathbf{u}^n)$. If the transfer is preserving, then the next time step can be taken with a preserving scheme, e.g. the scheme 
\begin{align*}
\dfrac{\mathbf{u}^{n+1} - \hat{\mathbf{u}}}{\Delta t} = S_{\mathbf{p}^{n+1}}(\hat{\mathbf{u}},\mathbf{u}^{n+1}) \overline{\nabla}\mathcal{I}_{\mathbf{p}^{n+1}}(\hat{\mathbf{u}},\mathbf{u}^{n+1}),
\end{align*}
which is preserving in the sense that 
\begin{align*}
\mathcal{I}_{\mathbf{p}^{n+1}}(\mathbf{u}^{n+1}) \! - \! \mathcal{I}_{\mathbf{p}^n}(\mathbf{u}^n) &= \mathcal{I}_{\mathbf{p}^{n+1}}(\mathbf{u}^{n+1}) - \mathcal{I}_{\mathbf{p}^{n+1}}(\hat{\mathbf{u}})\\
&= \left\langle \overline{\nabla}\mathcal{I}_{\mathbf{p}^{n+1}}(\hat{\mathbf{u}},\mathbf{u}^{n+1}), \mathbf{u}^{n+1} \!\! -\! \hat{\mathbf{u}} \right\rangle \\
&= \Delta t \left\langle \overline{\nabla}\mathcal{I}_{\mathbf{p}^{n+1}}(\hat{\mathbf{u}},\mathbf{u}^{n+1}), S_{\mathbf{p}^{n+1}}(\hat{\mathbf{u}},\mathbf{u}^{n+1}) \overline{\nabla}\mathcal{I}_{\mathbf{p}^{n+1}}(\hat{\mathbf{u}},\mathbf{u}^{n+1}) \right\rangle \\
& = 0,
\end{align*}
since $S_{\mathbf{p}^{n+1}}(\hat{\mathbf{u}},\mathbf{u}^{n+1})$ is skew-symmetric. If non-preserving transfer is used, corrections are needed in order to obtain a preserving numerical method. 
\begin{prop}
The scheme
\begin{align}
\mathbf{u}^{n+1}\! = \! \hat{\mathbf{u}}\! - \!\dfrac{(\mathcal{I}_{\mathbf{p}^{n+1}}(\hat{\mathbf{u}}) - \mathcal{I}_{\mathbf{p}^n}(\mathbf{u}^n))\mathbf{z}}{\left\langle \overline{\nabla}\mathcal{I}_{\mathbf{p}^{n+1}}(\hat{\mathbf{u}},\mathbf{u}^{n+1}), \mathbf{z} \right\rangle}\! + \! \Delta t S_{\mathbf{p}^{n+1}}(\hat{\mathbf{u}},\mathbf{u}^{n+1}) \overline{\nabla}\mathcal{I}_{\mathbf{p}^{n+1}}(\hat{\mathbf{u}},\mathbf{u}^{n+1}),
\label{eq:DGMM}
\end{align}
where $\mathbf{z}$ is an arbitrary vector chosen such that $\left\langle \overline{\nabla}\mathcal{I}_{\mathbf{p}^{n+1}}(\hat{\mathbf{u}},\mathbf{u}^{n+1}), \mathbf{z} \right\rangle \neq 0$, is first integral preserving in the sense that $ \mathcal{I}_{\mathbf{p}^{n+1}}(\mathbf{u}^{n\!+\!1})\! - \!\mathcal{I}_{\mathbf{p}^n}(\mathbf{u}^n) = 0$.
\end{prop}
\begin{proof}
\begin{align*}
 \mathcal{I}_{\mathbf{p}^{n+1}}(\mathbf{u}^{n\!+\!1})\! - \!\mathcal{I}_{\mathbf{p}^n}(\mathbf{u}^n) \! &= \! \mathcal{I}_{\mathbf{p}^{n+1}}(\mathbf{u}^{n+1}) - \mathcal{I}_{\mathbf{p}^{n+1}}(\hat{\mathbf{u}}) + \mathcal{I}_{\mathbf{p}^{n+1}}(\hat{\mathbf{u}}) - \mathcal{I}_{\mathbf{p}^n}(\mathbf{u}^n)\\
 &=  \!\left\langle \overline{\nabla}\mathcal{I}_{\mathbf{p}^{n+1}}(\hat{\mathbf{u}},\mathbf{u}^{n+1}), \mathbf{u}^{n+1} \!\! -\! \hat{\mathbf{u}} \right\rangle + \mathcal{I}_{\mathbf{p}^{n+1}}(\hat{\mathbf{u}}) - \mathcal{I}_{\mathbf{p}^n}(\mathbf{u}^n)\\
 &= \! \left\langle \!\! \overline{\nabla}\mathcal{I}_{\mathbf{p}^{n+1}}(\hat{\mathbf{u}},\mathbf{u}^{n+1}), \mathbf{u}^{n+1} \!\!-\! \hat{\mathbf{u}}\!  + \! \dfrac{(\mathcal{I}_{\mathbf{p}^{n+1}}(\hat{\mathbf{u}}) - \mathcal{I}_{\mathbf{p}^n}(\mathbf{u}^n))\mathbf{z}}{\left\langle \overline{\nabla}\mathcal{I}_{\mathbf{p}^{n+1}}(\hat{\mathbf{u}},\mathbf{u}^{n+1}), \mathbf{z} \right\rangle} \!\! \right\rangle  \!\\
 &= \! \Delta t \left\langle \!\! \overline{\nabla}\mathcal{I}_{\mathbf{p}^{n+1}}(\hat{\mathbf{u}},\mathbf{u}^{n+1}), S_{\mathbf{p}^{n+1}}(\hat{\mathbf{u}},\mathbf{u}^{n+1}) \overline{\nabla}\mathcal{I}_{\mathbf{p}^{n+1}}(\hat{\mathbf{u}},\mathbf{u}^{n+1}) \!\! \right\rangle  \!\\
 & = 0.
\end{align*}
The second equality follows from (\ref{eq:DG1}), the fourth equality from the scheme (\ref{eq:DGMM}), and the last equality follows from the skew-symmetry of $S_{\mathbf{p}^{n+1}}$.
\end{proof}
\noindent The correcting direction $\mathbf{z}$ should be chosen so as to obtain a minimal correction, and such that $\langle \overline{\nabla}\mathcal{I}_{\mathbf{p}^{n+1}}(\hat{\mathbf{u}},\mathbf{u}^{n+1}), \mathbf{z} \rangle \neq 0$. One possibility is simply taking $\mathbf{z} = \overline{\nabla}\mathcal{I}_{\mathbf{p}^{n+1}}(\hat{\mathbf{u}},\mathbf{u}^{n+1})$. In the FDM case one may alternatively choose $\mathbf{z} = D(\kappa)^{-1} \overline{\nabla}\mathcal{I}_{\mathbf{p}^{n+1}}(\hat{\mathbf{u}},\mathbf{u}^{n+1})$, and in the PUM case, $\mathbf{z} = A^{-1} \overline{\nabla}\mathcal{I}_{\mathbf{p}^{n+1}}(\hat{\mathbf{u}},\mathbf{u}^{n+1})$.

When using the PUM formulation, one may obtain a method for preserving transfer in the following manner. Any changes through e.g. $r$- $p$- and/or $h$-refinement between time steps will result in a change in the shape and/or number of basis functions. Denote by $\mathcal{B}^h = \mathrm{span}\{\varphi_i \}_{i=0}^M$ the trial space from the current time step and by $\hat{\mathcal{B}}^h = \mathrm{span}\{\hat{\varphi_i} \}_{i=0}^{\hat{M}}$ the trial space for the next time step, and note that in general, $M \neq \hat{M}$. We do not concern ourselves with how the new basis is found, but simply acknowledge that the basis changes through adaptivity measures as presented in e.g.$\!$ \cite{MM_Huang_Russell} or \cite{HPFEM_Babuska_Guo}. Our task is now to transfer the approximation $u^h$ from $ \mathcal{B}^h $ to $ \hat{\mathcal{B}}^h  $, obtaining an approximation $\hat{u}^h$, while conserving the first integral, i.e. $\mathcal{I}[u^h] = \mathcal{I}[\hat{u}^h]$. This can be formulated as a constrained minimization problem:
\begin{align*}
\min_{\hat{u}^h \in \tilde{\mathcal{B}^h}}||\hat{u}^h - u^h||_{L^2}^2 \quad \text{s.t.} \quad  \mathcal{I}[\hat{u}^h] = \mathcal{I}[u^h].
\end{align*}
We observe that
\begin{align*}
||\hat{u}^h - u^h||_{L^2}^2 &= \sum_{i=0}^{\hat{M}}\sum_{j=0}^{\hat{M}} \hat{u}_i \hat{u}_j\hat{A}_{ij} - 2\sum_{i=0}^{\hat{M}}\sum_{j=0}^M \hat{u}_i u_j^nC_{ij} + \sum_{i=0}^M\sum_{i=0}^M u_i^n u_j^n A_{ij}\\
 &= \mathbf{\hat{u}}^T \hat{A} \mathbf{\hat{u}} - 2 \mathbf{\hat{u}}^T C \mathbf{u}^n + \mathbf{u}^n A \mathbf{u}^n,
\end{align*}
where $A_{ij} = \langle \varphi_i,\varphi_j \rangle_{L^2}$, $\hat{A}_{ij} = \langle \hat{\varphi}_i,\hat{\varphi}_j \rangle_{L^2}$ and $C_{ij} = \langle \hat{\varphi}_i,\varphi_j \rangle_{L^2}$. Also observing that 
\begin{align*}
\mathcal{I}[\hat{u}^h] = \mathcal{I}_{\mathbf{p}^{n+1}}(\mathbf{\hat{u}}), \quad \mathcal{I}[u^h] = \mathcal{I}_{\mathbf{p}^{n}}(\mathbf{u}^n),
\end{align*}
the problem can be reformulated as 
\begin{align*}
\min_{\hat{\mathbf{u}} \in \mathbb{R}^{\hat{M}+1}} \mathbf{\hat{u}}^T \hat{A} \mathbf{\hat{u}} - 2 \mathbf{\hat{u}}^T C \mathbf{u}^n + \mathbf{u}^n A \mathbf{u}^n \quad \text{s.t.} \quad \mathcal{I}_{\mathbf{p}^{n+1}}(\mathbf{\hat{u}}) - \mathcal{I}_{\mathbf{p}^{n}}(\mathbf{u}^n) = 0.
\end{align*}
This is a quadratic minimization problem with one nonlinear equality constraint. Using the method of Lagrange multipliers, we find $\hat{\mathbf{u}}$ as the solution of the nonlinear system of equations
\begin{align*}
\hat{A}\hat{\mathbf{u}} - C \mathbf{u}^n - \lambda \nabla_{\hat{\mathbf{u}}} \mathcal{I}_{\mathbf{p}^{n+1}}(\hat{\mathbf{u}}) &= 0\\
\mathcal{I}_{\mathbf{p}^{n+1}}(\mathbf{\hat{u}}) - \mathcal{I}_{\mathbf{p}^{n}}(\mathbf{u}^n) &= 0,
\end{align*}
which can be solved numerically using a suitable nonlinear solver. 

In general, applicable also in the FDM case, given $\bar{\mathbf{u}}$ obtained by interpolating $\mathbf{u}^n$ onto $\mathbf{p}^{n+1}$ in a non-preserving manner, a preserving transfer operation is obtained by solving the system of equations
\begin{align*}
\hat{\mathbf{u}} - \bar{\mathbf{u}} - \lambda \nabla_{\hat{\mathbf{u}}} \mathcal{I}_{\mathbf{p}^{n+1}}(\hat{\mathbf{u}}) &= 0\\
\mathcal{I}_{\mathbf{p}^{n+1}}(\mathbf{\hat{u}}) - \mathcal{I}_{\mathbf{p}^{n}}(\mathbf{u}^n) &= 0.
\end{align*}

\subsection{Projection methods}

Let the function $f_\mathbf{p} : \mathbb{R}^M \times \mathbb{R}^M \rightarrow \mathbb{R}^M$ be such that
\begin{align}
\frac{\mathbf{u}^{n+1}-\mathbf{u}^n}{\Delta t} = f_\mathbf{p}(\mathbf{u}^{n},\mathbf{u}^{n+1})
\label{eq:arbit}
\end{align}
defines a step from time $t_n$ to time $t_{n+1}$ of any one-step method applied to (\ref{eq:pde_pure}) on the fixed grid represented by the discretization parameters $\mathbf{p}$. Then we define one step of an integral preserving linear projection method $\mathbf{u}^n \mapsto \mathbf{u}^{n+1}$ from $\mathbf{p}^n$ to $\mathbf{p}^{n+1}$ by
\begin{enumerate}
\item Interpolate $\mathbf{u}^n$ onto $\mathbf{p}^{n+1}$ by whatever means to get $\hat{\mathbf{u}}$,
\item Integrate $\hat{\mathbf{u}}$ one time step by computing $\tilde{\mathbf{u}} = \hat{\mathbf{u}} + {\Delta t}f_{\mathbf{p}^{n+1}}\left(\hat{\mathbf{u}},\tilde{\mathbf{u}}\right)$,
\item Compute $\mathbf{u}^{n+1}$ 
by solving the system of $M+1$ equations $\mathbf{u}^{n+1} = \tilde{\mathbf{u}} + \lambda \mathbf{z}$ and $\mathcal{I}_{\mathbf{p}^{n+1}}(\mathbf{u}^{n+1}) = \mathcal{I}_{\mathbf{p}^n}(\mathbf{u}^n)$, for $\mathbf{u}^{n+1} \in \mathbb{R}^M$ and $\lambda \in \mathbb{R}$, where the direction of projection $\mathbf{z}$ is typically an approximation to $\nabla_\mathbf{u} \mathcal{I}_{\mathbf{p}^{n+1}}(\mathbf{u}^{n+1})$.
\end{enumerate}

%
By utilizing the fact that for a method defined by (\ref{eq:arbit}) there exists an implicitly defined map $\Psi_{\mathbf{p}} : \mathbb{R}^M \rightarrow \mathbb{R}^M$ such that $\mathbf{u}^{n+1} = \Psi_{\mathbf{p}}\mathbf{u}^n$, we define
\begin{equation*}
g_\mathbf{p}(\mathbf{u}^n) := \frac{\Psi_{\mathbf{p}}\mathbf{u}^n-\mathbf{u}^{n}}{\Delta t},
\end{equation*}
and may then write the tree points above in an equivalent, more compact form as: Compute $\mathbf{u}^{n+1} \in \mathbb{R}^M$ and $\lambda \in \mathbb{R}$ such that
\begin{eqnarray}
\mathbf{u}^{n+1} - \hat{\mathbf{u}} - \Delta t g_{\mathbf{p}^{n+1}}\left(\hat{\mathbf{u}}\right) - \lambda \mathbf{z} &=& 0,
\label{linproj}\\
\mathcal{I}_{\mathbf{p}^{n+1}}(\mathbf{u}^{n+1}) - \mathcal{I}_{\mathbf{p}^{n}}(\mathbf{u}^n)&=& 0,
\label{preserved}
\end{eqnarray}
where $\hat{\mathbf{u}}$ is $\mathbf{u}^n$ interpolated onto $\mathbf{p}^{n+1}$ by an arbitrary procedure.

The following theorem and proof are reminiscent of Theorem 2 and its proof in \cite{norton13pma}, whose subsequent corollary shows how linear projection methods for solving ODEs are a subset of discrete gradient methods.
\begin{theo}
Let $g_\mathbf{p} : \mathbb{R}^M \rightarrow \mathbb{R}^M$ be a consistent discrete approximation of $f$ in (\ref{eq:pde_pure}) and let $\overline{\nabla}\mathcal{I}_{\mathbf{p}}(\mathbf{u}^{n},\mathbf{u}^{n+1})$ be any discrete gradient of the consistent approximation $\mathcal{I}_\mathbf{p}(\mathbf{u})$ of $\mathcal{I}\left[u\right]$ defined by (\ref{eq:Idef}) on the grid given by discretization parameters $\mathbf{p}$. If we set $S_{\mathbf{p}^{n+1}}$ in (\ref{eq:DGMM}) to be
\begin{equation}
S_{\mathbf{p}^{n+1}}(\hat{\mathbf{u}},\mathbf{u}^{n+1}) = \frac{g_{\mathbf{p}^{n+1}}(\hat{\mathbf{u}})\mathbf{z}^{\text{T}} - \mathbf{z} g_{\mathbf{p}^{n+1}}(\hat{\mathbf{u}})^{\text{T}}}{\left\langle \overline{\nabla}\mathcal{I}_{\mathbf{p}^{n+1}}\left(\hat{\mathbf{u}},\mathbf{u}^{n+1}\right), \mathbf{z} \right\rangle},
\label{eq:Sdef}
\end{equation}
then the linear projection method for solving PDEs on a moving grid, given by (\ref{linproj})-(\ref{preserved}), is equivalent to the discrete gradient method on moving grids, as given by (\ref{eq:DGMM}).
\end{theo}
\begin{proof}
For better readability, we set $\overline{\nabla}\mathcal{I} := \overline{\nabla}\mathcal{I}_{\mathbf{p}^{n+1}}\left(\hat{\mathbf{u}},\mathbf{u}^{n+1}\right)$. Assume that (\ref{linproj})-(\ref{preserved}) are satisfied. By applying (\ref{preserved}), we get that
\begin{align*}
\mathcal{I}_{\mathbf{p}^n}(\mathbf{u}^n)-\mathcal{I}_{\mathbf{p}^{n+1}}(\hat{\mathbf{u}}) &= \mathcal{I}_{\mathbf{p}^{n+1}}(\mathbf{u}^{n+1})-\mathcal{I}_{\mathbf{p}^{n+1}}(\hat{\mathbf{u}})\nonumber\\
&= \left\langle \overline{\nabla}\mathcal{I}, \mathbf{u}^{n+1} - \hat{\mathbf{u}} \right\rangle\\
&= \Delta t \left\langle \overline{\nabla}\mathcal{I}, g_{\mathbf{p}^{n+1}}(\hat{\mathbf{u}}) \right\rangle + \lambda \left\langle \overline{\nabla}\mathcal{I}, \mathbf{z} \right\rangle,
\end{align*}
and hence
\begin{align}
\lambda = \frac{\mathcal{I}_{\mathbf{p}^n}(\mathbf{u}^n)-\mathcal{I}_{\mathbf{p}^{n+1}}(\hat{\mathbf{u}})}{\left\langle \overline{\nabla}\mathcal{I}, \mathbf{z} \right\rangle} - \Delta t \frac{\left\langle \overline{\nabla}\mathcal{I}, g_{\mathbf{p}^{n+1}}(\hat{\mathbf{u}}) \right\rangle}{\left\langle \overline{\nabla}\mathcal{I}, \mathbf{z} \right\rangle}
\label{eq:lambda}
\end{align}
Substituting this into (\ref{linproj}), we get
\begin{align*}
\mathbf{u}^{n+1}=\hat{\mathbf{u}} + \frac{\mathcal{I}_{\mathbf{p}^n}(\mathbf{u}^n)-\mathcal{I}_{\mathbf{p}^{n+1}}(\hat{\mathbf{u}})}{\left\langle \overline{\nabla}\mathcal{I}, \mathbf{z} \right\rangle} \mathbf{z} + \Delta t \left(g_{\mathbf{p}^{n+1}}\left(\hat{\mathbf{u}}\right) - \frac{\left\langle \overline{\nabla}\mathcal{I}, g_{\mathbf{p}^{n+1}}(\hat{\mathbf{u}}) \right\rangle}{\left\langle \overline{\nabla}\mathcal{I}, \mathbf{z}\right\rangle} \mathbf{z} \right),
\end{align*}
where
\begin{align*}
g_{\mathbf{p}^{n+1}}\left(\hat{\mathbf{u}}\right) - \frac{\left\langle \overline{\nabla}\mathcal{I}, g_{\mathbf{p}^{n+1}}(\hat{\mathbf{u}}) \right\rangle}{\left\langle \overline{\nabla}\mathcal{I}, \mathbf{z} \right\rangle} \mathbf{z} &= \frac{\overline{\nabla}\mathcal{I}^{\text{T}}\mathbf{z} g_{\mathbf{p}^{n+1}}(\hat{\mathbf{u}}) - \overline{\nabla}\mathcal{I}^{\text{T}}g_{\mathbf{p}^{n+1}}(\hat{\mathbf{u}})\mathbf{z}}{\left\langle \overline{\nabla}\mathcal{I}, \mathbf{z} \right\rangle}\\
&= \frac{g_{\mathbf{p}^{n+1}}(\hat{\mathbf{u}})\mathbf{z}^{\text{T}}\overline{\nabla}\mathcal{I} - \mathbf{z} g_{\mathbf{p}^{n+1}}(\hat{\mathbf{u}})^{\text{T}}\overline{\nabla}\mathcal{I}}{\left\langle \overline{\nabla}\mathcal{I}, \mathbf{z} \right\rangle}
\end{align*}
and thus (\ref{eq:DGMM}) is satisfied, with $S_{\mathbf{p}^{n+1}}$ as given by (\ref{eq:Sdef}). Conversely, if $\mathbf{u}^{n+1}$ satisfies (\ref{eq:DGMM}), then (\ref{preserved}) is satisfied. Furthermore, inserting (\ref{eq:Sdef}) into (\ref{eq:DGMM}) and following the above deduction backwards, we get (\ref{linproj}), with $\lambda$ defined by (\ref{eq:lambda}).
\end{proof}

Since (\ref{eq:Sdef}) defines a particular set of choices for $S_{\mathbf{p}^{n+1}}$, the linear projection methods on moving grids constitute a subset of all possible discrete gradient methods on moving grids as defined by (\ref{eq:DGMM}). Note also that, since the linear projection methods are independent of the discrete gradient, each linear projection method defines an equivalence class of the methods (\ref{eq:DGMM}), uniquely defined by the choice of $g_{\mathbf{p}^{n+1}}$.

\subsection{Family of discretized integrals}
\label{sect:fiber}

At the core of the methods considered here is the notion that an approximation to the first integral $\mathcal{I}$ is preserved, and that this approximation is dependent on the discretization parameters which may change from iteration to iteration. That is, we have a family of discretized first integrals $\mathcal{I}_{\mathbf{p}}$, and at each time step the discretized first integral is exchanged for another. For each set of discretization parameters $\mathbf{p}$, there is a corresponding set of degrees of freedom $\mathbf{u}$, in which we search for a $\mathbf{u}$ such that $\mathcal{I}_{\mathbf{p}}(\mathbf{u})$ is preserved. This can be interpreted as a fiber bundle with base space $B$ as the set of all possible discretization parameters $\mathbf{p}$, and fibers $F_\mathbf{p}$ as the sets of all degrees of freedom such that the discretized first integral is equal to the initial discretized first integral, i.e. $F_\mathbf{p} = \{\mathbf{u} \in \mathbb{R}^M | \mathcal{I}_\mathbf{p}(\mathbf{u})  = \mathcal{I}_{\mathbf{p}^0}(\mathbf{u}^0) \}$. A similar idea, although without energy preservation, has been discussed by Bauer, Joshi and Modin in \cite{Bauer_Joshi_Modin}.

\section{Numerical experiments}

\subsection{General remarks on type of experiments made}
To provide examples of the application of our method and to investigate its accuracy, we have applied it to two one-dimensional PDEs: the sine-Gordon equation and the Korteweg--de Vries (KdV) equation. The choice of these equations were made because they both possess traveling wave solutions in the form of solitons, providing an ideal situation for $r$-adaptivity, which allows the grid points to cluster around wave fronts. The following experiments consider $r$-adaptivity only, and not $p$- or $h$-adaptivity. The sine-Gordon equation is solved using the FDM formulation of section \ref{sec:FDM}, while the KdV equation is solved using the PUM formulation of section \ref{sec:PUM}.

We wish to compare our methods to standard methods on fixed and adaptive meshes. This gives us four methods to consider: Fixed mesh methods with energy preservation by discrete gradients (DG), adaptive mesh methods with preservation by discrete gradients (DGMM), a non-preserving fixed grid method (MP), and the same method with adaptive mesh (MPMM). The former two methods are those described earlier in the paper, while the latter two are made differently for the two equations. In the sine-Gordon case, we use a finite difference scheme where spatial discretization is done using central finite differences and time discretization using the implicit midpoint rule. In the KdV case, the spatial discretization is performed the same way as for the discrete gradient schemes, while the time discretization is done using the implicit midpoint rule. The procedure for mesh adaptivity in the DGMM and MPMM schemes is presented in the next subsection.

The MPMM scheme for the sine-Gordon equation appeared unstable unless restrictively short time steps were used, and the results of those tests are therefore omitted from the following discussion. It is difficult to analyze the MPMM scheme and pinpoint an exact cause for this instability. However, it is worth noting that the other three schemes have preservation properties that should contribute to their stability; the DG and DGMM schemes have energy preservation properties, and the semidiscretization used for the sine-Gordon equation gives rise to a Hamiltonian system of equations which means that the MP scheme, which is symplectic, should perform well. On the other hand, the moving mesh strategy used breaks the symplecticity property in the MPMM scheme; specifically, the transfer strategies as presented in the next subsection do not preserve symplecticity. The results using MPMM for the KdV equation were better, and are presented.

\subsection{Adaptivity}
Concerning adaptivity of the mesh, we used a simple method for $r$-adaptivity which can be applied to both FDM and FEM problems in one spatial dimension. When applying moving mesh methods, one can either couple the evolution of the mesh with the PDE to be solved through a Moving Mesh PDE \cite{MMPDE_Huang} or use the rezoning approach, where function values and grid points are calculated in an intermittent fashion. Since our method is based on having a new set of grid points at each time step, and not coupling the evolution of the mesh to the PDE, the latter approach was used. It is based on an equidistribution principle, meaning that when $\Omega = [a,b]$ is split into $M$ intervals, one requires that
\begin{align*}
\int \limits_{x_i}^{x_{i+1}} \omega(x) \mathrm{d}x = \frac{1}{M} \int \limits_{a}^{b} \omega(x) \mathrm{d}x,
\end{align*}
where the monitor function $\omega$ is a function measuring how densely grid points should lie, based on the value of $u$. The choice of monitor function is problem dependent, and choosing it optimally may require considerable research. A variety of monitor functions have been studied for certain classes of problems, see e.g. \cite{BuddHuangRussell, BlomVerwer}. Through numerical experiments, we found little difference in performance when choosing between monitor functions based on arc-length and curvature, and have in the following used the former, that is, the generalized arc-length monitor function \cite{BuddHuangRussell}
\begin{align*}
\omega(x) = \sqrt{1 + k^2\left( \dfrac{\partial u}{\partial x}(x) \right)^2}.
\end{align*}
In this case, the equidistribution principle amounts to requiring that the weighted arc length (in the case $k=1$ one recovers the usual arc length) of $u$ over each interval is equal. In applications, we only have an approximation of $u$, meaning $\omega$ must be approximated as well; in our case, we have applied a finite difference approximation and obtained approximately equidistributing grids using de Boor's method as explained in \cite[pp.~36-38]{MM_Huang_Russell}.
We tried different smoothing techniques, including a direct smoothing of the monitor function and an iterative procedure for the regridding by De Boor's method (see e.g. \cite{MM_Huang_Russell, Pryce, Budd_Huang_Russell_blowup}). In the case of the KdV equation, there was little to no improvement using smoothing, but the sine-Gordon experiments showed significant improvement with direct smoothing; i.e., in De Boor's algorithm, we use the smoothed discretized monitor function
\begin{align*}
\bar{\omega}_i = \frac{\omega_{i-1} + 2 \omega_i + \omega_{i+1}}{4}.
\end{align*}

Having obtained the discretization parameters for the current time step, the numerical solution $\mathbf{u}$ from the previous time step must be transferred onto the new set of mesh points. We tested three different ways of doing this, two of which are using linear interpolation and cubic interpolation. The linear interpolation consists of constructing a function $\hat{u}(x)$ which is piecewise linear on each interval $[x_i^n, x_{i+1}^{n}]$ such that $\hat{u}(x_i^n) = u_i^n$, then evaluating this function at the new mesh points, giving the interpolated values $\hat{u}_i = \hat{u}(x_i^{n+1})$. The cubic interpolation consists of a similar construction, using cubic Hermite splines through the MATLAB function $\mathtt{pchip}$. Of these two transfer methods, the cubic interpolation yielded superior results in all cases, and so only results using cubic interpolation are presented. The third way, using preserving transfer as presented in section \ref{sect:mapping}, applies to the KdV example, where the PUM is used. Here, we found little difference between cubic interpolation and exact transfer, so results are presented using cubic interpolation for the transfer operation here as well. 

\subsection{Sine-Gordon equation}
The sine-Gordon equation is a nonlinear hyperbolic PDE in one spatial and one temporal dimension exhibiting soliton solutions, with applications in predicting dislocations in crystals and propagation of fluxons in junctions between superconductors. It is stated in initial value problem form as:
\begin{align}
&u_{tt} - u_{xx} + \sin(u) = 0, \quad (x,t) \in \mathbb{R} \times [0,T], \label{eq:SG} \\
&u(x,0) = f(x), \quad u_t(x,0) = g(x). \nonumber 
\end{align}
We consider a finite domain $[-L,L] \times [0,T]$ with periodic boundary conditions $u(-L) = u(L)$ and $u_t(-L) = u_t(L)$. The equation has the first integral
\begin{align*}
\mathcal{I}[u] = \int \limits_{\mathbb{R}}\frac{1}{2}u_t^2 + \frac{1}{2}u_x^2 + 1 - \cos(u) \mathrm{d}x .
\end{align*}
Introducing $v = u_t$, (\ref{eq:SG}) can be rewritten as a first-order system of PDEs:
\begin{align*}
\begin{bmatrix}
u_t\\
v_t
\end{bmatrix}
=
\begin{bmatrix}
v\\
u_{xx} - \sin(u)
\end{bmatrix},
\end{align*}
with first integral
\begin{align}
\mathcal{I}[u,v] = \int \limits_{\mathbb{R}}\frac{1}{2}v^2 + \frac{1}{2}u_x^2 + 1 - \cos(u) \mathrm{d}x.
\label{eq:SG_syst}
\end{align}
Finding the variational derivative of this, one can interpret the equation in the form (\ref{eq:pde_var}) with $S$ and $\frac{\delta \mathcal{I}}{\delta u}$ as follows:
\begin{align*}
S = 
\begin{bmatrix}
0 & 1 \\
-1 & 0
\end{bmatrix}, \quad
\dfrac{\delta \mathcal{I}}{\delta u}[u,v] =  
\begin{bmatrix}
\sin(u) - u_{xx} \\
v
\end{bmatrix}.
\end{align*}
We will apply the FDM approach presented in section \ref{sec:FDM}, approximating (\ref{eq:SG_syst}) by some quadrature with points $\{x_i\}_{i=0}^{M}$ and weights $\{\kappa_i\}_{i=0}^{M}$,
\begin{align*}
\mathcal{I}[u,v] \simeq \sum \limits_{i=0}^{M} \kappa_i\left(\frac{1}{2}v_i^2 + \frac{1}{2}u_{x,i}^2 + 1 - \cos(u_i)\right).
\end{align*}
In addition, we approximate the spatial derivatives with central differences. At the endpoints, a periodic extension is assumed, yielding the approximation
\begin{align*}
\mathcal{I}_{\mathbf{p}}(\mathbf{u}) = \sum \limits_{i=0}^{M} \kappa_i \left( \frac{1}{2}v_i^2 + \frac{1}{2}\left(\dfrac{\delta u_i}{\delta x_i} \right)^2 + 1 - \cos(u_i) \right).
\end{align*}
Here, $\delta w_i = w_{i+1} - w_{i-1}$ denotes central difference, with special cases $\delta u_0 = \delta u_{M} = u_{1} - u_{M-1}$, and $\delta x_0 = \delta x_{M} =  x_{1} - x_{0} + x_{M} - x_{M-1}$. Taking the gradient of $\mathcal{I}_{\mathbf{p}}(\mathbf{u})$ and applying the AVF discrete gradient gives
\begin{align*}
&\overline{\nabla}\mathcal{I}_{\mathbf{p}}(\mathbf{u}^n,\mathbf{u}^{n+1}) = \int \limits_{0}^{1} \nabla_\mathbf{u}\mathcal{I}_{\mathbf{p}}(\xi \mathbf{u}^{n} + (1-\xi)\mathbf{u}^{n+1}) \mathrm{d}\xi
\end{align*}
The periodic boundary conditions are enforced by setting $u_0 = u_M$. In the implementation, the $\kappa_i$ were chosen as the quadrature weights associated with the composite trapezoidal rule, i.e.
\begin{align*}
\kappa_0 = \dfrac{x_1 - x_0}{2},  \quad \kappa_{M} = \dfrac{x_{M} - x_{M-1}}{2}, \quad \kappa_i = \dfrac{x_{i+1} - x_{i-1}}{2}, \quad i = 1,...,M-1.
\end{align*}
Furthermore, $S$ was approximated by the matrix
\begin{align*}
S_d =
\begin{bmatrix}
0 & I \\
- I & 0
\end{bmatrix},
\end{align*}
with $I$ an $M \times M$ identity matrix. The exact solution considered was
\begin{align*}
u(x,t) = 4 \tan^{-1}\left( \dfrac{\sinh\left( \dfrac{ct}{\sqrt{1- c^2}} \right)}{c \cosh \left(\dfrac{x}{\sqrt{1-c^2}} \right)} \right). 
\end{align*}
This is a \textit{kink-antikink} system, an interaction between two solitons, each moving in different directions with speed $c \in (0,1)$, resulting in two wave fronts traveling in opposite directions. The wave fronts become steeper as $c \rightarrow 1$. 

\begin{figure}
        \centering
        \subfloat{
                \centering
                \includegraphics[width=0.49\textwidth]{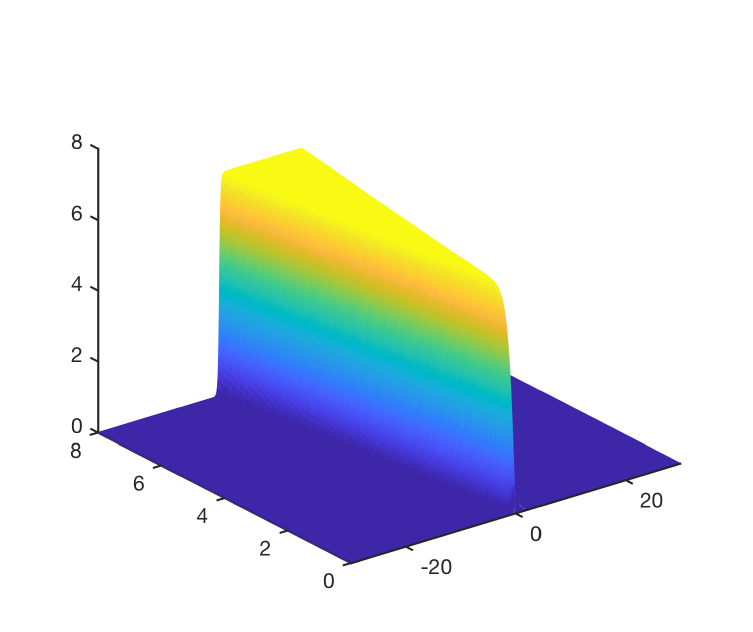}
        }
        \subfloat{
                \centering
                \includegraphics[width=0.49\textwidth]{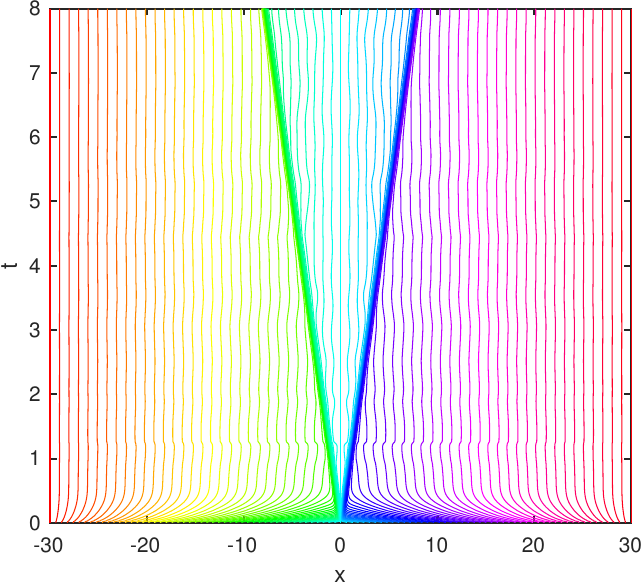}
        }
        \caption{\textit{Left}: Illustration of kink-antikink solution. \textit{Right}: Grid movement - each line represents the path of one grid point in time.}
        \label{fig:SG_fxn}
\end{figure}

\begin{figure}
        \centering
        \subfloat{
                \centering
                \includegraphics[width=0.49\textwidth]{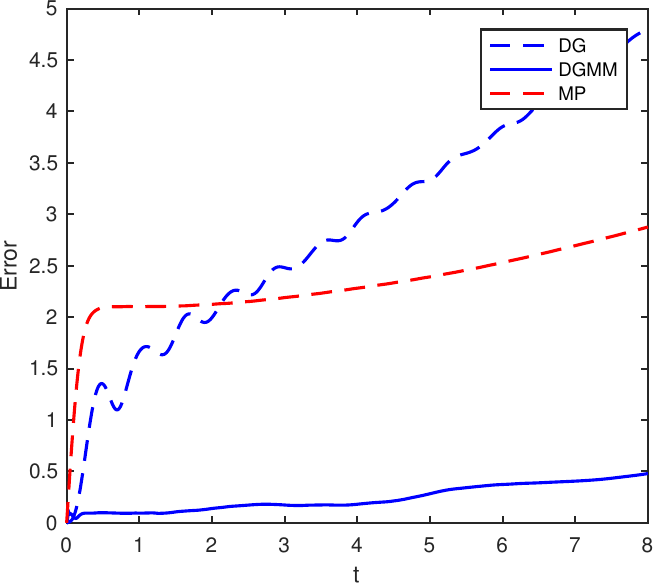}
        }
        \subfloat{
                \centering
                \includegraphics[width=0.49\textwidth]{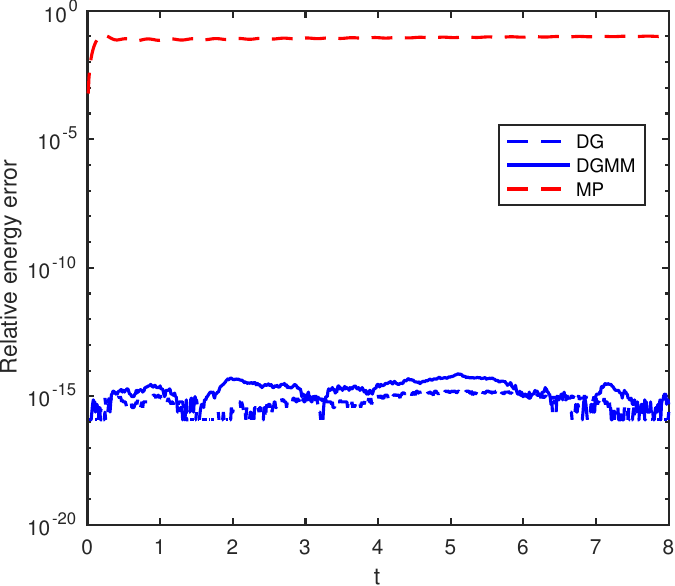}
        }
        \caption{\textit{Left}: $L_2$ error. \textit{Right}: Relative error in $I_{\mathbf{p}}$. \textit{Parameters}: $\Delta t = 0.01$, $M = 300$, $L = 30$, $c = 0.99$.}
        \label{fig:SG_globerr}
\end{figure}

Figure \ref{fig:SG_fxn} illustrates the analytical solution and shows the time evolution of the mesh as obtained with the DGMM method. Note that the grid points cluster along the wave fronts. The left hand side of Figure \ref{fig:SG_globerr} shows the time evolution of the error $E^u_n = ||u^I_n(x) - u(x,t_n)||_{L_2}$, where $u^I_n$ is a linear interpolant created from the pairs $(\mathbf{u}^n, \mathbf{x}^n)$. The right hand side of Figure \ref{fig:SG_globerr} shows the time evolution of the relative error in the discretized energy, $E^I_n = (I_{\mathbf{p}^n}(\mathbf{u}^n) - I_{\mathbf{p}^0}(\mathbf{u}^0))/I_{\mathbf{p}^0}(\mathbf{u}^0)$. We can see that the long-term behaviour of the MP scheme is superior to that of the DG scheme, but when mesh adaptivity is applied, the DGMM scheme is clearly better. Also note that while the DG and DGMM schemes preserve $I_{\mathbf{p}}$ to machine precision, the MP scheme does not.

\begin{figure}
        \centering
        \subfloat{
                \centering
                \includegraphics[width=0.49\textwidth]{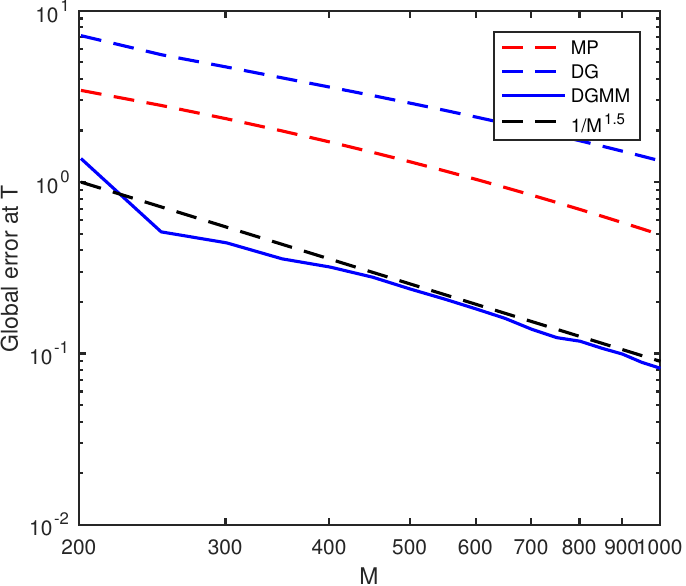}
        }
        \subfloat{
                \centering
                \includegraphics[width=0.49\textwidth]{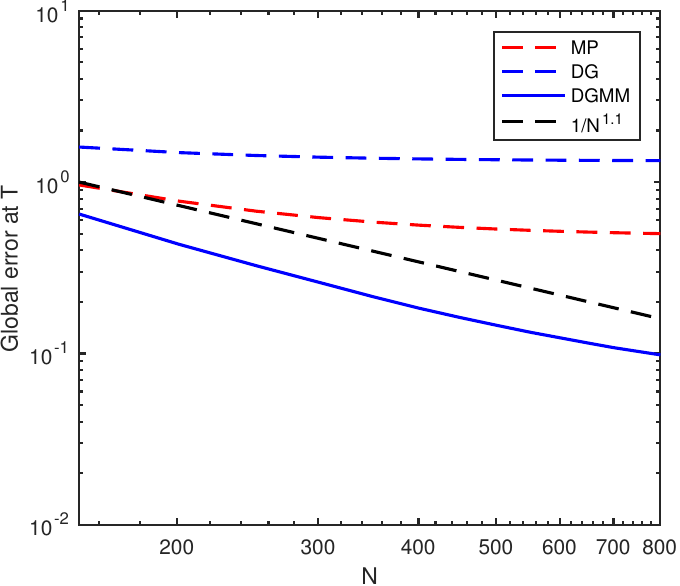}
        }
        \caption{\textit{Left}: Error at $T = 8$ as a function of $M$, with $\Delta t = 0.008$, $c = 0.99$, $L = 30$. \textit{Right}: Error at $T = 8$ as a function of $N = T/\Delta t$, with $M = 1000$, $c = 0.99$, $L = 30$.}
        \label{fig:SG_convrates}
\end{figure}

\begin{figure}
        \centering
        \includegraphics[width=0.49\textwidth]{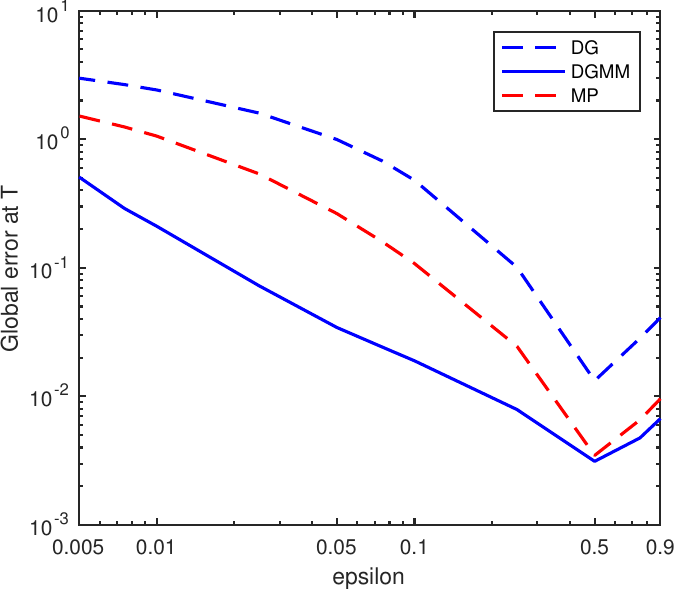}
        \caption{Error at $T = 8$ as a function of $\varepsilon$, with $\Delta t = 0.01$, $M = 600$ and $L = 30$.}
        \label{fig:SG_epsilon}
\end{figure}

Figure \ref{fig:SG_convrates} shows the convergence behaviour of the three schemes with respect to the number of spatial discretization points $M$, and the number of time steps $N$. Note that the DG and MP methods plateau at $N \simeq 400$; this is due to the error stemming from spatial discretization dominating the time discretization error for these methods, while the DGMM scheme has lower spatial discretization error. The convergence order of the DGMM scheme was measured using a first order polynomial fitting of $\log(E_n^u)$ to $\log(M)$ and $\log(N)$. The convergence order with respect to $M$ was calculated as 1.518, and the convergence order with respect to $N$ was measured at 1.121. 

Finally, to illustrate the applicability of the DGMM scheme to harder problems, Figure \ref{fig:SG_epsilon} shows the error at stopping time of the methods as a function of a parameter $\epsilon$ representing the increasing speed of the solitons ($c = 1 - \varepsilon$). From this plot, it is appararent that while the non-adaptive MP scheme is competitive at low speeds, the moving mesh method provides significantly more accuracy as $c \rightarrow 1$.

\subsection{Korteweg--de Vries equation}
The KdV equation is a nonlinear PDE with soliton solutions modelling shallow water surfaces, stated as
\begin{align}
u_t + u_{xxx} + 6u u_x = 0.
\label{eq:KdV}
\end{align}
It has infinitely many first integrals, one of which is the Hamiltonian
\begin{align*}
\mathcal{H}[u] = \int \limits_{\mathbb{R}} \frac{1}{2}u_x^2 - u^3 \mathrm{d}x.
\end{align*}
With this Hamiltonian, we can  write (\ref{eq:KdV}) in the form (\ref{eq:pde_var}) with
$S$ and $\frac{\delta \mathcal{H}}{\delta u}$ as follows:
\begin{align*}
S = \dfrac{\partial}{\partial x} , \quad
\dfrac{\delta \mathcal{H}}{\delta u}[u] =  - u_{xx} - 3u^2.
\end{align*}
We will apply the PUM approach to create a numerical scheme which preserves an approximation to $\mathcal{H}[u]$, splitting $\Omega = [-L,L]$ into $M$ elements $\{[x_i, x_{i+1}]\}_{i=0}^{M-1}$ and using Lagrangian basis functions $\varphi_j$ of arbitrary degree for the trial space. Approximating $u$ by $u^h$ as in section \ref{sec:PUM}, we find 
\begin{align}
\mathcal{H}_{\mathbf{p}}(\mathbf{u}) &= \mathcal{H}[u^h] = \int_{\Omega} \frac{1}{2}(u^h_x)^2 - (u^h)^3 \mathrm{d}x \nonumber \\
&= \frac{1}{2} \sum \limits_{j,k} u_j u_k \int_{\Omega} \varphi_{j,x}\varphi_{k,x} \mathrm{d}x - \sum \limits_{j,k,l}   u_j u_k u_l \int_{\Omega} \varphi_{j}\varphi_{k}\varphi_{l} \mathrm{d}x.
\label{eq:approxH}
\end{align}
The integrals can be evaluated exactly and efficiently by considering elementwise which basis functions are supported on the element before applying Gaussian quadrature to obtain exact evaluations of the polynomial integrals. We define
\begin{align*}
D_{i j k} = \int_{\Omega} \varphi_{i}\varphi_{j}\varphi_{k} \mathrm{d}x \quad \mathrm{ and } \quad E_{i j} =  \int_{\Omega} \varphi_{i,x}\varphi_{j,x} \mathrm{d}x.
\end{align*}
The matrices $A$ and $B$ with
\begin{align*}
A_{ij} =  \int_{\Omega} \varphi_{i}\varphi_{j} \mathrm{d}x \quad \mathrm{ and } \quad  B_{ji} = \int_{\Omega} \varphi_{i}\varphi_{j,x} \mathrm{d}x 
\end{align*} 
are formed in the same manner. Note that $B$ is in this case independent of $\mathbf{u}$. Applying the AVF method yields the discrete gradient 
\begin{align*}
\overline{\nabla}\mathcal{H}_{\mathbf{p}}(\mathbf{u}^n,\mathbf{u}^{n+1}) = \int \limits_{0}^{1} \nabla_\mathbf{u}\mathcal{H}_{\mathbf{p}}(\xi \mathbf{u}^{n} + (1-\xi)\mathbf{u}^{n+1}) \mathrm{d}\xi
\end{align*}
such that, with the convention of summation over repeated indices,
\begin{align*}
(\overline{\nabla}\mathcal{H}_{\mathbf{p}})_i  = \dfrac{1}{2} E_{ij}(u_j^n + u_j^{n+1}) -   D_{ijk}(u_j^n (u_k^n  + \frac{1}{2} u_k^{n+1}) + u_j^{n+1}(\frac{1}{2} u_k^n  + u_k^{n+1})).
\end{align*}
This gives us all the required terms for forming the system (\ref{eq:PUM_ODE}) and applying the discrete gradient method to it. During testing, the $\varphi_j$ were chosen as piecewise linear polynomials. The exact solution considered is of the form
\begin{align}
u(x,t) = \dfrac{c}{2} \mathrm{sech}^{2}\left( \dfrac{\sqrt{c}}{2} (x-ct) \right),
\label{eq:KdV_exact}
\end{align}
which is a right-moving soliton with $c$ as the propagation speed, chosen as $c=6$ in the numerical tests. We have considered periodic boundary conditions 
on a domain $\left[-L,L\right] \times [0,T]$, with $L = 100$ in all the following results.

Our discrete gradient method on a moving mesh (DGMM) is compared to the same method on a static, equidistributed mesh (DG), and the implicit midpoint method on static (MP) and moving mesh (MPMM). The spatial discretization is performed the same way in all cases. Figure \ref{fig:uplots} shows an example of exact and numerical solutions at $t = 15$. Note that the peak in the exact solution will be located at $x = c t$.

\begin{figure}[ht]
                \centering
                \includegraphics[width=0.9\textwidth]{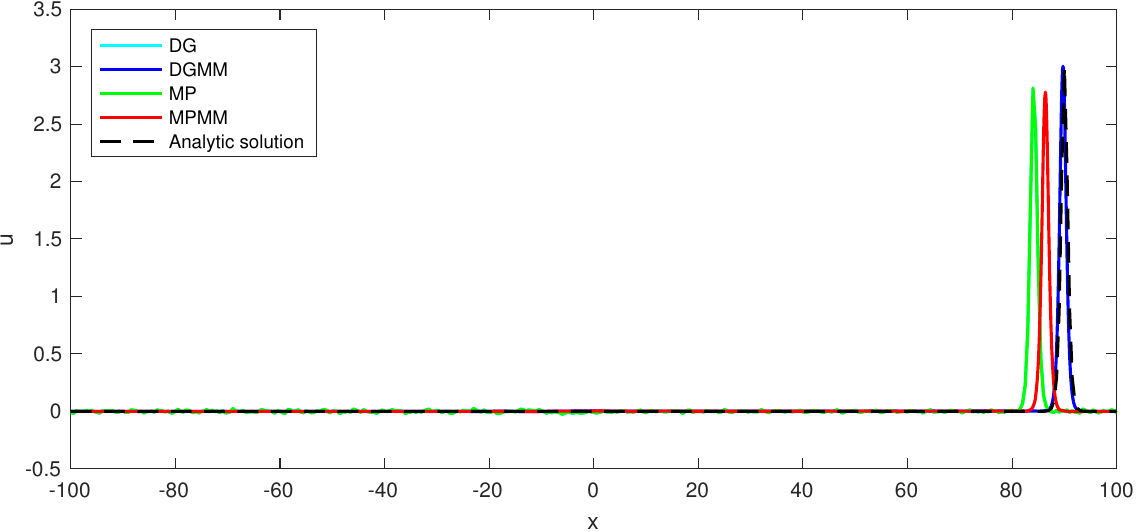}
        \caption{Solutions at T = 15. $\Delta t = 0.01$, $M = 400$. MP and DG are almost indistinguishable.}
        \label{fig:uplots}
\end{figure}

\noindent To evaluate the numerical solution, it is reasonable to look at the distance error
\begin{align*}
E^{\text{dist}}_n = ct_n - x^*,
\end{align*}
where $x^* = \text{arg} \max \limits_{x} u_h(x,t_n) $, i.e. the location of the peak in the numerical solution. Another measure of the error is the shape error
\begin{align*}
E^{\text{shape}}_n = \left|\left|u_h(x,t_n) - u \left(x,\frac{x^*}{c} \right) \right| \right|,
\end{align*}
where the peak of the exact solution is translated to match the peak of the numerical solution, and the shapes of the solitons are compared.

Figure \ref{fig:KdV_energy} confirms that the DG and DGMM methods preserve the approximated Hamiltonian (\ref{eq:approxH}), while it is also worth noting that in the case of the midpoint method, the error in this conserved quantity is much larger on a moving than on a static mesh. Similar behaviour is also observed for a moving-mesh method for the regularized long wave equation in the recent paper \cite{Lu_Huang_Qiu}, where it is concluded that a moving mesh method with a conservative property would be an interesting research topic. Figure \ref{fig:KdV_overtime}, where the phase and shape errors are plotted up to $T=15$, is an example of how the DGMM method performs comparatively better with increasing time.

\begin{figure}
        \centering
        \subfloat{
                \centering
                \includegraphics[width=0.49\textwidth]{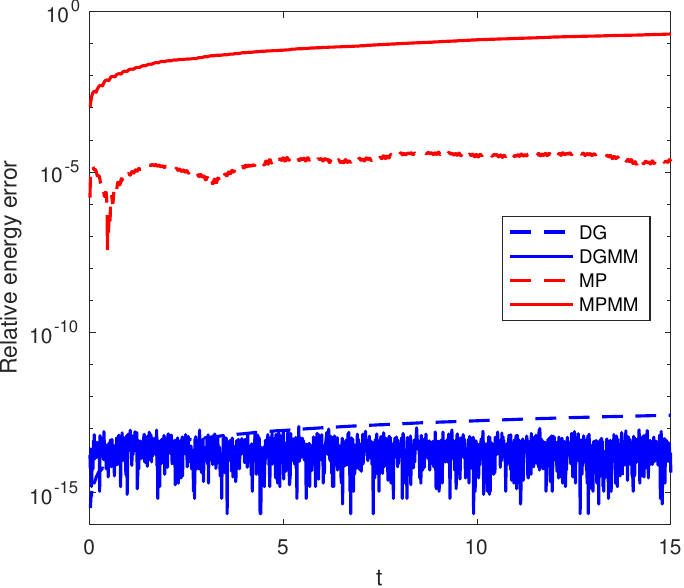}
        }
        \caption{Relative error in the Hamiltonian plotted as a function of time $t \in \left[0,15\right]$. $\Delta t = 0.01$, $M = 400$.}
        \label{fig:KdV_energy}
\end{figure}

\begin{figure}
        \centering
        \subfloat{
                \centering
                \includegraphics[width=0.49\textwidth]{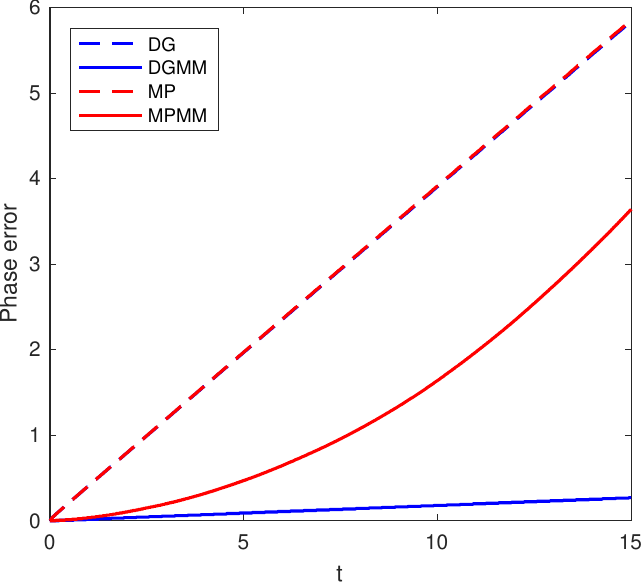}
        }
        \subfloat{
                \centering
                \includegraphics[width=0.49\textwidth]{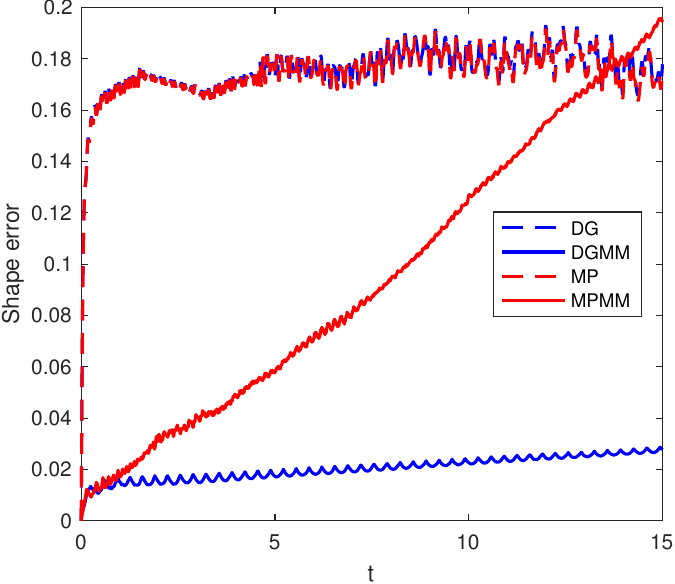}
        }
        \caption{Phase error (left) and shape error (right) as a function of time. $\Delta t = 0.01$, $M = 400$.}
        \label{fig:KdV_overtime}
\end{figure}
In figures \ref{fig:KdV_byM} and \ref{fig:KdV_byN} we present the phase and shape errors for the different methods as a function of the number of elements $M$ and the number of time steps $N$, respectively. 
Reference lines are included to give an indication of the rate of convergence. We also calculated this for the DGMM method by first degree polynomial fitting of the error curve, giving a convergence order of $1.135$ for the phase error and $2.311$ for the shape error as a function of $M$. As a function of $N$, we get a convergence order of $1.492$ for the phase error, and $1.609$ for the shape error (the latter measured up to $N = 320$, where it flattens out). 
We observe that the DGMM scheme performs especially well, compared to the other three schemes, for a coarse spatial discretization compared to the discretization in time. In figure \ref{fig:KdV_byc}, the phase and shape errors are plotted as a function of the parameter $c$ in the exact solution (\ref{eq:KdV_exact}), where we note that $\frac{c}{2}$ is the height of the wave; increasing $c$ leads to sharper peaks and thus a harder numerical problem. As expected, the advantages of the DGMM method is less evident for small $c$, but we observe that the DGMM method outperforms the static grid midpoint method already when $c=2$.

\begin{figure}
        \centering
        \subfloat{
                \centering
                \includegraphics[width=0.49\textwidth]{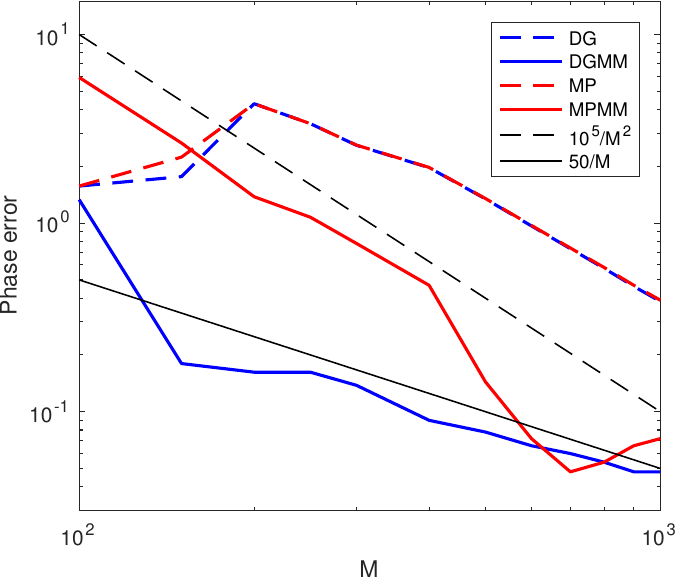}
        }
        \subfloat{
                \centering
                \includegraphics[width=0.49\textwidth]{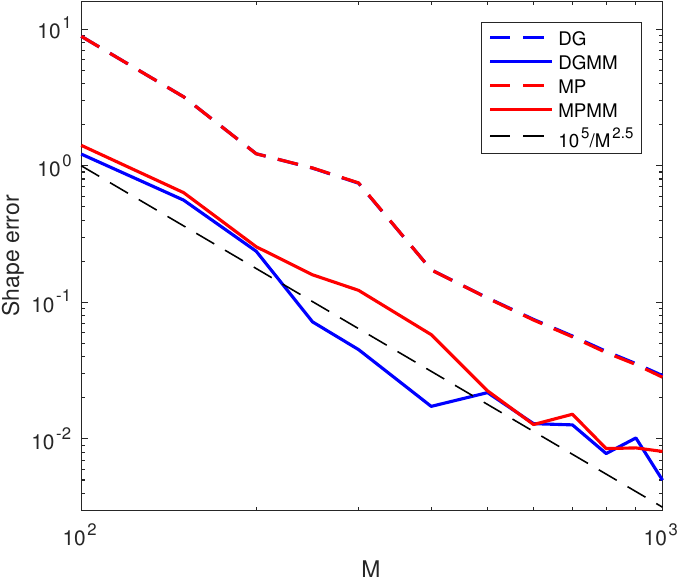}
        }
        \caption{Phase error (left) and shape error (right) as a function of the number of elements $M$, at time $T = 5$. $\Delta t = 0.01$.}
        \label{fig:KdV_byM}
\end{figure}

\begin{figure}
        \centering
        \subfloat{
                \centering
                \includegraphics[width=0.49\textwidth]{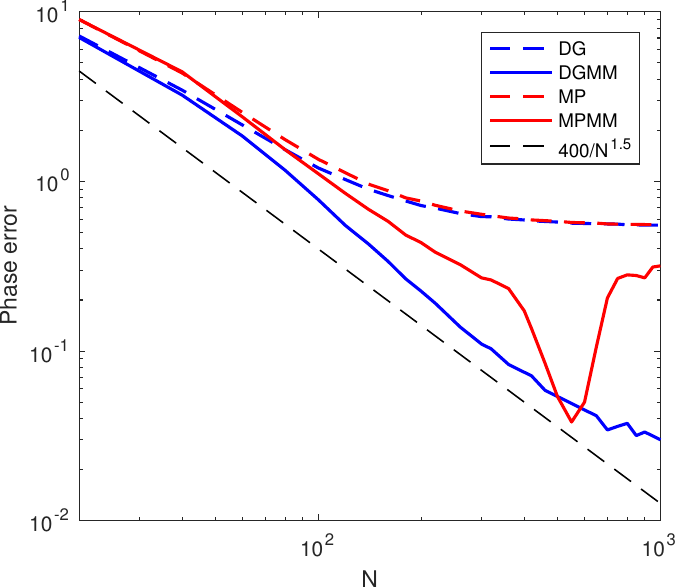}
        }
        \subfloat{
                \centering
                \includegraphics[width=0.49\textwidth]{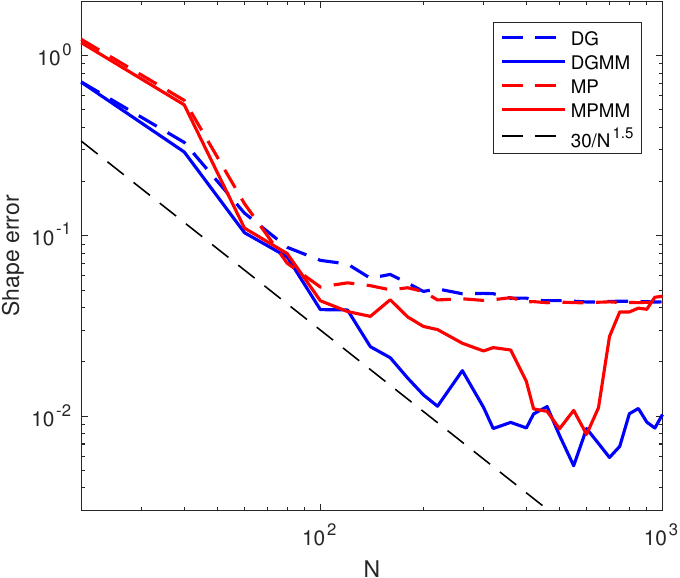}
        }
        \caption{Phase error (left) and shape error (right) at time $T = 5$, as a function of the number of time steps $N = T/\Delta t$. $M = 800$.}
        \label{fig:KdV_byN}
\end{figure}

\begin{figure}
        \centering
        \subfloat{
                \centering
                \includegraphics[width=0.49\textwidth]{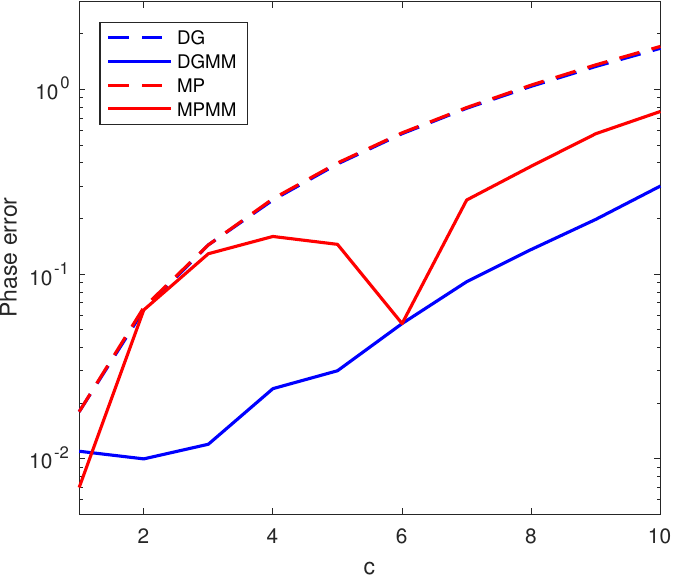}
        }
        \subfloat{
                \centering
                \includegraphics[width=0.49\textwidth]{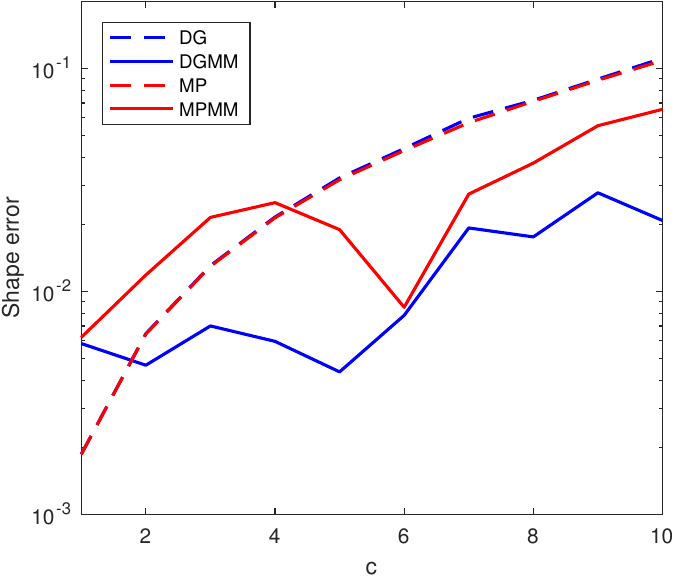}
        }
        \caption{Phase error (left) and shape error (right) as a function of $c$ in the exact solution (\ref{eq:KdV_exact}), at time $t = 5$. $\Delta t = 0.01$, $M = 800$.}
        \label{fig:KdV_byc}
\end{figure}

\subsection{Execution time}
The code used is not optimized, so any quantitative comparison to standard methods has not been performed; it is still possible to make some qualitative observations. Adding adaptivity increases time per iteration slightly since the systems become more complicated, especially in the case of the PUM approach where the matrices $A$ and $B$ need to be recalculated, at each time step when adaptivity is used. This increases runtime somewhat when compared to fixed grid methods. However, adaptivity allows for using fewer degrees of freedom, and so decreases the degrees of freedom needed for a given level of accuracy. This accuracy gain is more pronounced the harder the problem is (steeper wave fronts etc.), and so it stands to reason that there will be situations where adaptive energy preserving methods will outperform non-adaptive and/or non-preserving methods. This is in accordance with what we have observed from our not optimized experiments.

\section{Conclusion}

In this paper, we have introduced a general framework for producing adaptive first integral preserving methods for partial differential equations. This is done by first providing two means of producing first integral preserving methods on arbitrary fixed grids, then showing how to extend these methods to allow for adaptivity while preserving the first integral. Numerical testing shows that moving mesh methods coupled with discrete gradient methods provide good solvers for the sine-Gordon and Korteweg--de Vries equations. It would be of interest to apply the method to higher-dimensional PDEs with a more challenging geometry, preferably using the PUM approach, to investigate its accuracy as compared to conventional methods, and to test whether $h$- and/or $p$-refinement provides a notable improvement. It may also prove fruitful to explore the ideas presented in \cite{Bauer_Joshi_Modin} to make the transfer operations between sets of discretization parameters in a more natural setting than simply interpolating, as suggested in section \ref{sect:fiber}. Furthermore, analysis of the methods considered here could provide important insight into e.g. stability, consistency and convergence order. 

\bibliography{EP,tb}
\bibliographystyle{ieeetr}

\end{document}